\documentclass[a4paper,11pt]{article}
\usepackage{amsmath}
\usepackage{amsthm}
\usepackage{amssymb}
\usepackage{mathtools}
\usepackage{mathrsfs}
\usepackage{physics}
\usepackage{stmaryrd} 
\usepackage{bm} 
\usepackage{graphicx} 
\usepackage{subcaption}
\usepackage{xspace}
\usepackage[margin=1.0in]{geometry}
\usepackage[hidelinks]{hyperref}

\newcommand\B{\mathbb B}
\newcommand\N{\mathbb N}

\newcommand\R{\mathbb R}
\newcommand\PP{\mathbb P}
\newcommand\E{\mathbb E}

\newcommand\cA{\mathcal A}
\newcommand\cN{\mathcal N}
\newcommand\cP{\mathcal P}
\newcommand\cQ{\mathcal Q}

\let\oldnl\nl
\newcommand{\nonl}{\renewcommand{\nl}{\let\nl\oldnl}}

\newcommand\Ind{\boldsymbol 1}

\newcommand\Lip{\mathrm{Lip}}

\DeclarePairedDelimiter{\floor}{\lfloor}{\rfloor}

\newcommand{\PR}[1]{\PP\left(#1\right)}

\newcommand{\EX}[1]{\E\left[#1\right]}

\newcommand{\RomanNumeralCaps}[1]
    {\MakeUppercase{\romannumeral #1}}

\usepackage{xcolor}

\newcommand\Sym{\mathtt{Sym}}

\theoremstyle{plain}
\newtheorem{theorem}{Theorem}[section]
\newtheorem*{theorem*}{Theorem}
\newtheorem{lemma}[theorem]{Lemma}
\newtheorem*{lemma*}{Lemma}
\newtheorem{corollary}[theorem]{Corollary}
\newtheorem*{corollary*}{Corollary}
\newtheorem{proposition}[theorem]{Proposition}
\newtheorem*{proposition*}{Proposition}

\newtheorem*{fact*}{Fact}

\theoremstyle{definition}

\newtheorem*{definition*}{Definition}

\newtheorem*{problem*}{Problem}

\newtheorem*{example*}{Example}

\theoremstyle{remark}
\newtheorem{remark}[theorem]{Remark}
\newtheorem*{remark*}{Remark}

\allowdisplaybreaks

\title{Hessian descent for spherical spin glasses with uniform log-Sobolev disorder}
\author{
    Fu-Hsuan Ho\thanks{Department of Mathematics, Weizmann Institute of Science, 76100 Rehovot, Israel}
    }
\date{\today}

\begin{document}

\maketitle

\begin{abstract}
    The present work concerns spherical spin glass models with disorder satisfying a uniform logarithmic Sobolev inequality. We show that the Hessian descent algorithm introduced in~\cite{following21} can be extended to this setting, thanks to the abundance of small eigenvalues near the edge of the Hessian spectrum. Combined with the ground state universality recently proven in~\cite{sawhney_free_2024}, this implies that when the model is in the full-RSB phase, the Hessian descent algorithm can find a near-minimum with high probability.
    
    Our proof consists of two main ingredients. First, we show that the empirical spectral distribution of the Hessian converges to a semicircular law via the moment method. Second, we use the logarithmic Sobolev inequality to establish concentration and obtain uniform control of the spectral edge.
\end{abstract}

\section{Introduction} \label{sec:intro}

Spherical spin glasses are random polynomial functions---called Hamiltonians---defined on high-dimensional spheres, where the coefficients are typically taken to be i.i.d. standard Gaussian variables. These models have served as a natural testing ground for studying optimization problems involving high-dimensional, non-convex objective functions. In a seminal work \cite{following21}, Subag introduced the Hessian descent algorithm, which proceeds as follows: starting from the origin of the Euclidean ball, the algorithm follows, at each step, the direction corresponding to the most negative eigenvalue of the Hamiltonian's Hessian matrix. Once the path reaches the boundary sphere, the algorithm outputs the final point. When the model satisfies the so-called full-RSB condition, Subag showed that, with high probability, this procedure finds a point on the sphere whose Hamiltonian value is near the global minimum.

The goal of the present work is to generalize Subag's result to spherical spin glasses with non-Gaussian disorder, specifically those whose coefficients satisfy a uniform logarithmic Sobolev inequality. Since the algorithm descends along directions associated with the left edge of the spherical Hessian spectrum, it is crucial to ensure that the spectrum indeed has many small (negative) eigenvalues throughout the descent. Our main result, Theorem~\ref{thm:edge}, shows that this spectral edge behavior persists with high probability under the log-Sobolev condition. On the other hand, a recent result of Sawhney and Sellke \cite{sawhney_free_2024} establishes the universality of the ground state energy, demonstrating that, under mild moment assumptions, the minimum value of the Hamiltonian is asymptotically the same as in the Gaussian case. Putting these pieces together, we conclude that for full-RSB models with log-Sobolev disorder, the Hessian descent algorithm still finds near-optimal configurations with high probability---just as in the Gaussian case.

\subsection{Definitions and result} \label{sec:def and result}

Let $\N$ be the set of integers, and $\llbracket a,b\rrbracket$ be the set of integers between $a$ and $b$. For $N\in\N$, let the collection $J=(J^{(p)}_{i_1,\ldots,i_p})_{p\geq 2,i_1,\ldots,i_p\in\llbracket 1,N\rrbracket}$ consist of independent random variables with mean $0$ and variance $1$. 
Moreover, we assume the following uniform logarithmic Sobolev condition throughout this paper:
\begin{enumerate}
    \item[\hypertarget{ass:LS}{(LS)}] \emph{Uniform logarithmic Sobolev bound}.
    There exists a constant $K>0$ such that for all $p\geq 2$, $i_1,\ldots,i_p\in\llbracket 1, N\rrbracket$ and differentiable function $f:\R\rightarrow\R$ with $\int_\R f(x)^2\dd{x}=1$,
    \begin{align}
	\EX{f(J^{(p)}_{i_1,\ldots,i_p})^2\log(f(J^{(p)}_{i_1,\ldots,i_p})^2)} \leq 2K \EX{f'(J^{(p)}_{i_1,\ldots,i_p})^2}.
        \label{eq:LSI d=1}
    \end{align}
\end{enumerate}
Define the $(N-1)$-sphere and the $N$-ball by
\begin{align*}
{\bm{S}}^{N-1} = \{\bm{x}\in\R^N:\norm{\bm{x}}=\sqrt{N}\},
\quad\quad
\bm{B}^N = \{\bm{x}\in\R^N:\norm{\bm{x}}\leq \sqrt{N}\},
\end{align*}
where $\norm{\cdot}$ is the Euclidean norm on $\R^N$.
For a sequence of real numbers $\gamma = (\gamma_p)_{p=2}^\infty$, the spherical spin glass Hamiltonian
on $\bm{B}^{N}$ is the random function\footnote{
We extend the usual domain $\bm{S}^{N-1}$ to the $\bm{B}^N$ because our goal is to generalize the gradient descent algorithm in \cite{following21}. This algorithm constructs a path from the origin of $\bm{B}^N$ to some point $\bm{x}\in \bm{S}^{N-1}$, which requires information in the whole ball $\bm{B}^N$.
}
$H_N:\bm{B}^{N}\rightarrow\R$ 
\begin{align}
    H_N(\bm{x})
    =
    \sum_{p=2}^\infty 
	\frac{\gamma_p}{N^{(p-1)/2}}\sum_{i_1,\ldots,i_p=1}^N J_{i_1,\ldots,i_p}^{(p)} x_{i_1}\cdots x_{i_p}
	=
	\sum_{p=2}^\infty 
    \gamma_p
    H_{N}^{(p)}(\bm{x}).
    \nonumber
\end{align}
When required, we write $H_N^{(p)}(\bm{x};J^{(p)})$ to show dependence of $H_N^{(p)}$ on $J^{(p)}$. Note that $H_N(\bm{x})$ has  covariance
\begin{align}
\EX{H_N(\bm{x})H_N(\bm{x'})} = N\xi\left(\frac{\langle\bm{x},\bm{x'}\rangle}{N}\right),  
\quad 
\xi(z) = \sum_{p=2}^\infty \gamma_p^2 z^p,
\label{eq:covariance}
\end{align}
where $\langle\cdot,\cdot\rangle$ is the inner product on $\R^N$. 
We assume that the power series $\xi$ has the radius of convergence greater than $1$.

Given $\bm{x}\in\bm{B}^{N}$, the Euclidean Hessian matrix of $H_N(\bm{x})$ is defined as 
\begin{align}
    \nabla_E^2 H_{N}(\bm{x})
    \coloneqq
    \sum_{p=2}^\infty 
    \gamma_p
    \nabla_E^2 H_{N}^{(p)}(\bm{x}), \nonumber
\end{align}
where $\nabla_E^2 H_{N}^{(p)}(\bm{x})$ is the Euclidean Hessian matrix of $H_{N}^{(p)}(\bm{x})$ defined entrywise as 
\begin{align}
    \nabla_E^2 H_{N}^{(p)}(\bm{x})_{jk} 
    = 
    \frac{1}{N^{(p-1)/2}}
    \sum_{i_1,\ldots,i_p=1}^N J_{i_1,\ldots,i_p} 
    \partial_{j}\partial_{k} 
    (x_{i_1}\cdots x_{i_p}), 
    \quad 
    j,k\in\llbracket 1,N\rrbracket.
    \nonumber
\end{align}
The Hessian matrix of the Hamiltonian $H_N(\bm{x})$ on the ball $\bm{B}^{N}$ at $\bm{x}$ is defined as
\begin{align}
    \nabla^2 H_{N}(\bm{x}) 
    =
    P_{\bm{x}}^\intercal\nabla_E^2 H_{N}(\bm{x}) P_{\bm{x}}
    =
    \sum_{p=2}^\infty \gamma_p P_{\bm{x}}^\intercal\nabla_E^2 H_{N}^{(p)}(\bm{x})P_{\bm{x}}
    =
    \sum_{p=2}^\infty \gamma_p \nabla^2 H_{N}^{(p)}(\bm{x}), \label{eq:projected Hessian}
\end{align}
where $P_{\bm{x}} = P_{\bm{x}}^\intercal = I_N - \frac{\bm{x} \bm{x}^\intercal}{\norm{\bm{x}}^2}$
is the orthogonal projection matrix onto the orthogonal space of $\bm{x}$.
We denote by $\lambda_{N,\bm{x}}^{(1)},\ldots,\lambda_{N,\bm{x}}^{(N)}$ the eigenvalues of $\nabla^2 H_N(\bm{x})$, and we denote by $\mu_{N,\bm{x}} = \frac{1}{N}\sum_{i=1}^N \delta_{\lambda_{N,\bm{x}}^{(i)}}$ the empirical spectral distribution of $\nabla^2 H_N(\bm{x})$, and by $\bar{\mu}_{N,\bm{x}}$ the mean empirical spectral distribution of $\nabla^2 H_N(\bm{x})$.

The main result of this paper is the following uniform control on the edge.
\begin{theorem}[Uniform control on the edge]
    \label{thm:edge}
    Suppose Condition~\hyperlink{ass:LS}{(LS)} holds.
    Then, for all $\varepsilon\in (0,1)$, there exist $\delta>0$ and $c>0$ such that
    \begin{align*}
        \PP
        \Big(
        \forall \bm{x}\in\bm{B}^N : 
        \mu_{N,\bm{x}}((-\infty,-2\xi''(\rho_{\bm{x}})^{1/2}+\varepsilon])\geq \delta 
        \Big) \geq 1-e^{-cN},
    \end{align*}
    where $\xi''(z)=\sum_{p=2}^\infty \gamma_p^2p(p-1)z^{p-2}$ and $\rho_{\bm{x}}=\norm{\bm{x}}^2/N$.
\end{theorem}

\subsection{Implications for the Hessian descent algorithm}

As mentioned in the beginning of the introduction, for Gaussian Hamiltonians, i.e., the $J$'s are i.i.d. standard Gaussian random variables, Subag introduced (See Theorem 1.4 in \cite{following21}) a Hessian descent algorithm that outputs with high probability a point $\bm{x}$ such that $\frac{1}{N}H_N(\bm{x})\approx E_{\nabla^2}(\xi)$, where
\begin{align*}
    E_{\nabla^2}(\xi) \coloneqq \int_0^1 \xi''(t)^{1/2} \dd{t}.
\end{align*}
For models satisfying the condition\footnote{This condition is equivalent to saying that the Hamiltonian is of full-RSB by the work of Chen and Panchenko \cite{chen_temperature_2017}, Jagannath and Tobasco \cite{jagannath_bounds_2018} and Talagrand \cite{talagrand_free_2006}.} 
that $\xi''(q)^{-1/2}$ is concave on $(0,1]$, Chen and Sen showed in \cite{chen_parisi_2017} that $E_{\nabla^2}$ coincides with the ground state energy $E_\star$. That is,
\begin{align*}
    E_{\nabla^2}(\xi) 
    = \lim_{N\rightarrow\infty} \frac{1}{N}\mathbb{E}
    \min_{\bm{x}\in\bm{S}^{N-1}}
    H_N(\bm{x})
    \stackrel{\mathrm{a.s.}}{=} \lim_{N\rightarrow\infty} \frac{1}{N}
    \min_{\bm{x}\in\bm{S}^{N-1}}
    H_N(\bm{x}) \eqqcolon E_\star(\xi).
\end{align*}
Therefore, the Hessian descent can find a point $\bm{x}$ such that $\frac{1}{N}H_N(\bm{x})\approx E_\star(\xi)$.

Back to our setting.  For the Hessian descent to output with high probability a point $\bm{x}$ such that $\frac{1}{N}H_N(\bm{x})\approx E_{\nabla^2}$, 
one only needs to replace (3.12) in the proof of Theorem~1.4 in \cite{following21} with Theorem 1.1.
Moreover, in a recent work of Sawhney and Sellke \cite{sawhney_free_2024}, they showed that (Theorem 1.5 in~\cite{sawhney_free_2024}) under finite $2p$-moment conditions on the $p$-th disorders (which is much weaker than Condition~\hyperlink{cond:LS}{(LS)}), the ground state converges to the one of Gaussian Hamiltonians. In particular, for models having $\xi''(q)^{-1/2}$ concave on $q\in (0,1]$, the gradient descent can find a point $\bm{x}$ with $\frac{1}{N}H_N(\bm{x})\approx E_{\nabla^2}(\xi) = E_\star(\xi)$, as in the Gaussian setting. 

\subsection{Proof strategy} \label{sec:outline}

We follow the same strategy as in \cite{following21} to prove Theorem~\ref{thm:edge}. First, we discretize $\bm{B}^N$ into an exponential number of points. Then, for each $\bm{x}$ belonging to the discretization, we prove that the edge of the empirical spectral distribution of $\nabla H_N(\bm{x})$ satisfies the desired property with high probability, and then estimate the probability in Theorem~\ref{thm:edge} via the union bound. 

Before describing how we adapt the strategy above to our setting, we first note that Condition~\hyperlink{ass:LS}{(LS)} has the following standard implications, with the same constant $K$. 
\begin{enumerate}
    \item[\hypertarget{ass:SG}{(SG)}] \emph{Sub-Gaussianity}. 
    For all $s\in\R$, $p\geq 2$ and $N\in\N$,
    \begin{align}
        \EX{\exp(s J_{i_1,\ldots,i_p}^{(p)})} \leq 
        \exp(\frac{K s^2}{2}). \label{eq:MGF bnd}
    \end{align}
    \item[\hypertarget{ass:M}{(M)}] \emph{Existence of moments}.
    For all $p\geq 2$ and $k\in\N$, there exists a constant $C(p,k,K)>0$ such that 
    \begin{align*}
     \sup_{r\in\llbracket 2,p\rrbracket}\sup_{i_1,\ldots,i_r\in\llbracket 1, N\rrbracket} \EX{\abs*{J_{i_1,\ldots,i_r}^{(r)}}^k} \leq C(p,k,K) <\infty.
    \end{align*}
\end{enumerate}
Condition~\hyperlink{ass:M}{(M)} follows directly from Condition~\hyperlink{ass:SG}{(SG)}. 
Condition~\hyperlink{ass:SG}{(SG)} can be derived as follows. Condition~\hyperlink{ass:LS}{(LS)} and the Herbst argument (see the last paragraph in Page 94 of \cite{ledoux_concentration_2005}) yield for any Lipschitz function $f:\R\rightarrow\R$ with Lipschitz constant $\abs*{f}_\Lip$, 
     \begin{align*}
         \log \EX{\exp(sf(J^{(p)}_{i_1,\ldots,i_p}))}
         \leq 
         \frac{Ks^2\abs*{f}_\Lip^2}{2} 
         + s \EX{ f(J^{(p)}_{i_1,\ldots,i_p})}.
     \end{align*}
     The bound \eqref{eq:MGF bnd} then follows from taking $f(x)=x$ in the previous inequality.

We explain how Condition~\hyperlink{ass:LS}{(LS)}, Condition~\hyperlink{ass:SG}{(SG)} and Condition~\hyperlink{ass:M}{(M)} come into play to adapt the precedent work \cite{following21} to the present setting.
\begin{enumerate}
    \item In the Gaussian setting as \cite{following21}, one can show that $\mu_{N,\bm{x}}$ is a semicircular law plus one extra atom at $0$. This is because the law of the Hamiltonian is isotropic, and the proof goes by computing the Hessian for the north pole. This argument does not work in our setting. However, since we only need to control the edge of the spectrum, it suffices to establish: 1) the mean empirical spectral distribution $\bar{\mu}_{N,\bm{x}}$ of $\nabla H_N(\bm{x})$ converges weakly to the semicircular law and 2) for a given test function $f$, $\int f\dd{\mu_{N,\bm{x}}}$ concentrates to $\int f\dd{\bar{\mu}_{N,\bm{x}}}$. The first point is proven using the moment method, which relies only on Condition~\hyperlink{ass
}{(M)}. In contrast, the second point requires the full strength of Condition~\hyperlink{ass
}{(LS)} to achieve exponential concentration.
    \item The proof of Lemma~3.1 of \cite{following21} (where the proof is referred to as the one of Lemma C.1 of \cite{arousGeometryTemperatureChaos2020}) requires Dudley's entropy bound (see Corollary 5.25~in \cite{van_handel_probability_2014}) and the Borell--TIS inequality to control the supremum of $H^{(p)}_N(\bm{x})$. Dudley's entropy bound remains valid under Condition~\hyperlink{ass:SG}{(SG)}, while Condition~\hyperlink{ass:LS}{(LS)} is required to replace the Borell--TIS inequality by the concentration inequality for measure satisfying the logarithmic Sobolev inequality (see Proposition~2.3 of \cite{ledoux_concentration_1999}).
\end{enumerate}

\paragraph{Outline.} In Section~\ref{sec:Lip continuity}, we provide the regularity estimates for the Hamiltonian $H_N(\bm{x})$. In Section~\ref{sec:truncated model} and Section~\ref{sec:proof of T upp bnd}, we prove that the mean empirical spectral distribution of a truncated Hamiltonian (i.e., we take the partial sum in \eqref{eq:projected Hessian} up to some $p_1$) converges to a semicircular law. In Section~\ref{sec:concentration}, we provide a uniform control of the edge of the Hessian spectrum. Finally, we combine everything from before in Section~\ref{sec:main proof} to prove Theorem~\ref{thm:edge}.

\subsection{Notation} \label{sec:notation}
For a real number $r\geq 1$, the $r$-norm of $\bm{x}\in\R^N$ is defined by $\norm{\bm{x}}_r=(\sum_{i=1}^N\abs{x_i}^r)^{1/r}$, and we adopt the abbreviation $\norm{\bm{x}}=\norm{\bm{x}}_2$ for the Euclidean norm. 
We will make use of the fact that $\norm{\bm{x}}_r\leq \norm{\bm{x}}$ for all $r\geq 2$.

For a Lipschitz function $f:\R\rightarrow\R$, we denote by $\abs*{f}_\Lip$ the Lipschitz constant of $f$, and we define $\Lip_1(\R)$ to be the set of bounded Lipschitz functions $f:\R\rightarrow\R$ with $\abs*{f}_\Lip=1$. For probability measures $\mu_1,\mu_2$ on $\R$, we  
\begin{align*}
    \norm*{\mu_1-\mu_2}_{LB} 
    = \sup_{f\in\Lip_1(\R)}
    \abs{
    \int_{\R} f(y)\dd{\mu_1(y)}
    -
    \int_{\R} f(y)\dd{\mu_2(y)}
    }.
\end{align*}
Finally, for $\rho>0$, we denote by $\mu_{sc,\rho}$ the semicircular law with the radius parameter $\rho$, which is defined by the density
\begin{align*}
    \dd\mu_{sc,\rho}(y) = \frac{2}{\pi \rho^2}\sqrt{\rho^2 - y}\Ind_{\abs{y}\leq \rho}\dd{y}.
\end{align*}

\section{Regularity of the Hamiltonian} \label{sec:Lip continuity}


The goal of this section is to provide regularity estimates of the Hamiltonian $H_N(\bm x)$. We start with estimates for the pure $p$-spin models.
\begin{proposition}
\label{prop:Lipschitz} 
    Let $K>0$ be as in Condition~\hyperlink{ass:LS}{(LS)} and $C>0$ be an absolute. 
    Then, the followings hold. 
    \begin{itemize}
    \item For all $p\geq 2$ and $N\in\N$, 
    \begin{align}
        \PR{\sup_{\bm{x}\in\bm{B}^{N}}
    \norm*{\nabla^2_E H_{N}^{(p)}(\bm{x})}_{op}
    \geq C\sqrt{K}p^{3/2}(p-1)
    }
    \leq \exp(-\frac{C^2p^3(p-1)^2}{2}N).
        \label{eq:p Hessian Op}
    \end{align}
    \item For all $i=1,2,3$, $p\geq \max\{2,i\}$ and $N\in\N$,
    \begin{align}
        &
        \PR{
        \exists \bm{x}\in\bm{B}^N,
        \exists \bm{v}\in\R^N
        \ \text{with} \ \norm{\bm{v}}=1:
        \abs*{\partial_{\bm{v}}^i H_{N}^{(p)}(\bm{x})} 
        \geq \frac{2C\sqrt{K p}}{N^{(i-2)/2}}\frac{p!}{(p-i)!}
        } \nonumber \\
        &
        \leq \exp(-\frac{C^2p}{2}N),
        \label{eq:p directional}
    \end{align}
     where $\partial^i_{\bm{v}}$ is the $i$-th order directional derivative in direction $\bm{v}\in\R^N$ defined by
    \begin{align}
        \partial_{\bm{v}}^i H_{N}^{(p)}(\bm{x})
        =
        \frac{1}{N^{(p-1)/2}}
        \sum_{j_1,\ldots,j_i=1}^N 
        v_{j_1}\cdots v_{j_i}
        \sum_{i_1,\ldots,i_p=1}^N 
        J^{(p)}_{i_1,\ldots,i_p}
        \partial_{j_1}\cdots\partial_{j_i}(x_{i_1}\cdots x_{i_p}).
        \label{eq:directional der}
    \end{align}
    \item For all $p\geq 3$ and $N\in\N$,
    \begin{align}
        &
        \PR{
        \exists \bm{x},\bm{y}\in\bm{B}^N :
        \norm{\nabla^2_E H_{N}(\bm{x})-\nabla^2_E H_{N}(\bm{y})}_{op} 
        \geq \frac{2C\sqrt{K p}}{\sqrt{N}}\frac{p!}{(p-3)!}\norm{\bm{x}-\bm{y}}
        } \nonumber \\
        &
        \leq \exp(-\frac{C^2p}{2}N). \label{eq:p Hessian Lip}
    \end{align}
    \end{itemize}
\end{proposition}
The proof of the proposition will be given in Section \ref{sec:proof of Lipschitz}, after we prove the auxiliary Lemma \ref{lem:sup estimate}. Before tending to it,   we provide the following immediate corollary on the regularity of $H_N(\bm{x})$.
\begin{corollary}
    \label{cor:Lipschitz}
    There exist constants $R>0$ and $c>0$ depending only on $\xi$ and $K$ such that for all $i=1,\ldots,3$ and $N\in\N$,
    \begin{align}
        &
        \PP
        \Big(
        \forall \bm{x}\in \bm{B}^N,\,\forall \bm{v}\in\R^N \  \text{with} \ \norm{\bm{v}}=1 : \abs{\partial_{\bm{v}}^i H_N(\bm{x})} < R N^{1-\frac{i}{2}}
        \Big)
        \geq 1-e^{-cN},
        \label{eq:directional}
        \\
        &
        \PP
        \Big(
        \forall \bm{x},\mathbf{y}\in \bm{B}^N : \norm*{\nabla_E^2 H_N(\bm{x}) - \nabla_E^2 H_N(\mathbf{y})}_{op} < \frac{R}{\sqrt{N}} \norm{\bm{x}-\mathbf{y}}
        \Big)
        \geq 1-e^{-cN}. \label{eq:Hessian Lip}
    \end{align}
\end{corollary}
\begin{proof}
    Let  $i\in\{1,2,3\}$ and note that
 \begin{align*}
     \abs*{\partial_{\bm{v}}^i H_{N}(\bm{x})}
     \leq \sum_{p=\max\{2,i\}}^\infty \gamma_p\abs*{\partial_{\bm{v}}^i H_{N}^{(p)}(\bm{x})}.
 \end{align*}
 Let $R_i=2C\sqrt{K}\sum_{p\geq \max\{2,i\}}\gamma_p\sqrt{p}\frac{p!}{(p-i)!}$ where $C$ is the same absolute constant as in Lemma~\ref{lem:sup estimate}. By the union bound and \eqref{eq:p directional}~in~Proposition~\ref{prop:Lipschitz},
 \begin{align}
     &\PR{
     \exists \bm{x}\in\bm{B}^N,
     \exists \bm{v}
     \ \text{with} \ \norm{\bm{v}}=1:
     \abs{\partial_{\bm{v}}^i H_{N}(\bm{x})} 
     \geq \frac{R_i}{N^{(i-1)/2}}
     } \nonumber \\
     &\leq 
     \sum_{p=\max\{2,i\}}^\infty \exp(-\frac{C^2p}{2}N)
     \leq \exp(-cN), \nonumber
 \end{align}
 for sufficiently large $N$. By adjusting $c$ if needed,  the same bound on the probability holds for general $N$.

It remains to prove \eqref{eq:Hessian Lip}. For any $\bm{x},\bm{y}\in\bm{B}^N$ with $\bm{x}\neq \bm{y}$, setting $\bm{u}:=(\bm{x}-\bm{y})/\|\bm{x}-\bm{y}\|$, we have 
\begin{align}
    \norm*{\nabla_E^2 H_N(\bm{x}) - \nabla_E^2 H_N(\bm{y})}_{op}&=\max_{\|\bm{v}\|=1} | \partial_{\bm{v}}^2 
     H_N(\bm{x}) - \partial_{\bm{v}}^2 H_N(\bm{y})| \nonumber \\
  &   = 
     \max_{\|\bm{v}\|=1} \Big|\int_{0}^{\|\bm{x}-\bm{y}\|} \partial_{\bm{u}}\partial_{\bm{v}}^2 
     H_N(\bm{x}+t\bm{u}) dt\Big| \nonumber \\
     &\leq 
     \left(
     \max_{\|\bm{v}\|=1}
     \max_{\bm{z}\in\bm{B}^N}|\partial_{\bm{v}}^3 
     H_N(\bm{z})|
     \right) \norm{\bm{x}-\bm{y}}
     . \label{eq:third order derivative}
\end{align}
Then, applying \eqref{eq:directional} with $i=3$, \eqref{eq:third order derivative} yields 
\begin{align*}
     &\PR{
     \exists \bm{x},\bm{y}\in\bm{B}^N:
     \norm{\nabla^2_E H_{N}(\bm{x})-\nabla^2_E H_{N}(\bm{y})}_{op} 
     \geq \frac{R_3}{\sqrt{N}}\norm{\bm{x}-\bm{y}}
     } \nonumber \\
     &\leq \PR{
     \exists \bm{x},\bm{y}\in\bm{B}^N:
     \left(
     \max_{\|\bm{v}\|=1}
     \max_{\bm{z}\in\bm{B}^N}|\partial_{\bm{v}}^3 
     H_N(\bm{z})|
     \right) \norm{\bm{x}-\bm{y}} 
     \geq \frac{R_3}{\sqrt{N}}\norm{\bm{x}-\bm{y}}
     } \nonumber \\
     &=
     \PR{
     \max_{\|\bm{v}\|=1}
     \max_{\bm{z}\in\bm{B}^N}|\partial_{\bm{v}}^3 
     H_N(\bm{z})|
     \geq 
     \frac{R_3}{\sqrt{N}}
     }
     \leq e^{-cN},
 \end{align*}
 where $c$ is the same constant as in \eqref{eq:directional}. Therefore, the proof is completed by taking $R = \max\{R_1,\ldots,R_3\}$.
\end{proof}

We shall need the following in the proof of Proposition \ref{prop:Lipschitz}. 
\begin{lemma}
    \label{lem:sup estimate}
    Recall that $K$ is the  constant from Condition~\hyperlink{ass:LS}{(LS)}.
    There exists an absolute constant $C>0$ such that 
    for all $N\in\N$,   
    \begin{align}
    \frac{1}{N}\EX{\sup_{\bm{x}\in\bm{B}^{N}}\abs*{H_{N}(\bm{x})}}
    \leq C\sqrt{K\xi(1)\log_+\frac{\xi'(1)}{\xi(1)}}, \label{eq:sup estimate.1} 
\end{align}
    where $\log_+ x=\max\{1,\log x\}$.
    Moreover, for all integer $p\geq 2$, integer $N\in\N$ and $t>0$, we have 
    \begin{align}
        \PR{
        \frac{1}{N}\sup_{\bm{x}\in\bm{B}^{N}}\abs*{H_{N}^{(p)}(\bm{x})}\geq 
        2C\sqrt{Kp}
        }\leq \exp(-\frac{C^2p}{2}N).
        \label{eq:sup estimate.2}
    \end{align}
\end{lemma}

\begin{remark}
The work of Sawhney and Sellke showed (Theorem 1.5 in~\cite{sawhney_free_2024}) a finer universality result than \eqref{lem:sup estimate} for spin glass Hamiltonians with finite mixture. More precisely, under finite $2p$-moment conditions on the $p$-th disorders (which is weaker than Condition~\hyperlink{cond:LS}{(LS)}), the ground state converges to that of Gaussian Hamiltonians. However, here we want a \emph{non-asymptotic} estimate of the ground state.
\end{remark}

\begin{proof}[Proof of Lemma~\ref{lem:sup estimate}]
    We start with the proof of \eqref{eq:sup estimate.1} and we follow the proof of the upper bound of Proposition C.1 in \cite{montanari_solving_2024}.  
    Denoting by $\B^N=\{\bm{x}\in\R^N:\norm*{\bm{x}}\leq 1\}$ the $N$-dimensional unit ball, 
    we introduce the normalized Hamiltonian 
    \begin{align}
        \hat{H}_N(\bm{x}) = \frac{1}{\sqrt{NK}} H_N(\sqrt{N}\bm{x}), 
        \qquad \bm{x}\in \B^N.
    \end{align}
    The canonical (pseudo) metric of $\hat{H}_N(\cdot)$ is given by
    \begin{align*}
    d(\bm{x},\bm{y})
    &:= \EX{(\hat{H}_N(\bm{x})-\hat{H}_N(\bm{y}))^2}^{1/2}
    = \sqrt{2(\xi(1)-\xi(\langle \bm{x},\bm{y}\rangle))},
    \qquad \bm{x},\bm{y}\in\B^N.
    \end{align*}  
    Using the independence of the coefficients $(J^{(p)}_{i_1,\ldots,i_p})_{p\geq 2,i_1,\ldots,i_p\in \llbracket 1,N\rrbracket}$ and Condition \hyperlink{ass:SG}{(SG)}, one can check the sub-Gaussian estimate
      \begin{align*}
      \EX{\exp(s(\hat{H}_N(\bm{x})-\hat{H}_N(\bm{y})))}
    	\leq 
    	\exp(\frac{s^2}{2}d(\bm{x},\bm{y})^2),
        \qquad \bm{x},\bm{y}\in \B^N.
    \end{align*}
    Of course, the process $\hat{H}_N(\cdot)$ on $\B^N$ is  centered and separable (see Definition~5.22~in~\cite{van_handel_probability_2014}).  Hence, by 
    Dudley's entropy bound (cf. Corollary~5.25~in~\cite{van_handel_probability_2014}), there exists an absolute constant $C>0$ such that
    \begin{align}
        \EX{\sup_{\bm{x}\in\B^N}\hat{H}_{N}(\bm{x})}
        \leq 
        12\int_0^{2\xi(1)} \sqrt{\log \cN(\B^N,d,\varepsilon)} \dd{\varepsilon}
        \leq 
        C\sqrt{N\xi(1)\log_+\frac{\xi'(1)}{\xi(1)}},
        \label{eq:dudley bound}
    \end{align}
    where $\cN(\B^N,d,\varepsilon)$ is the covering number of $\B^N$ with respect to the canonical metric of $\hat{H}_N(\cdot)$, and the second inequality in \eqref{eq:dudley bound} above follows from the estimate of the metric entropy \[\log\cN(\B^N,d,\varepsilon)\]
    in the paragraph after (39) in the proof of Proposition C.1 in \cite{montanari_solving_2024} which does not require $\hat{H}(\cdot)$ to be a Gaussian process. Therefore, with the same absolute constant $C>0$ as in \eqref{eq:dudley bound},
    \begin{align}
        \frac{1}{N}\EX{\sup_{\bm{x}\in\bm{B}^N} H_N(\bm{x})}
        =
        \frac{\sqrt{K}}{\sqrt{N}}\EX{\sup_{\bm{x}\in\B^N}\hat{H}_{N}(\bm{x})}
        \leq 
        C\sqrt{K\xi(1)\log_+\frac{\xi'(1)}{\xi(1)}}.
        \label{eq:sup estimate done.1}
    \end{align}
    By replacing $H_N(\cdot)$ with $-H_N(\cdot)$ in the argument above to derive \eqref{eq:sup estimate done.1}, we obtain \eqref{eq:sup estimate.1}.
    
    It remains to prove \eqref{eq:sup estimate.2}.
    Note that for any real-valued functions $g_1$ and $g_2$ defined on a set $S$, we have
\begin{align*}
    \sup_S g_1 - \sup_S g_2 \leq \sup_S(g_1-g_2).
\end{align*}
Moreover, we have $H^{(p)}_{N}(\bm{y};J^{(p)})-H^{(p)}_{N}(\bm{y};{J'}^{(p)})=H^{(p)}_{N}(\bm{y};J^{(p)}-{J'}^{(p)})$, where we use the notation $H^{(p)}_{N}(\bm{y};J^{(p)})$ to make the dependence in the disorder explicit.  
Thus, we have
\begin{align}
    \bigg|
    \sup_{\bm{y}\in\bm{B}^N} 
    \abs*{H^{(p)}_{N}(\bm{y};J^{(p)})} - 
    \sup_{\bm{y}\in\bm{B}^N} 
    \abs*{H^{(p)}_{N}(\bm{y};{J'}^{(p)})}
    \bigg|
    &\leq 
    \sup_{\bm{y}\in\bm{B}^N} \abs*{H^{(p)}_{N}(\bm{y};J^{(p)}-{J'}^{(p)})}, \label{eq:sup |H|}
\end{align}
where the inequality above follows from the element inequality $\abs{\abs{x}-\abs{y}}\leq \abs{x-y}$ for all $x,y\in\R$. 

On the other hand, for all $\bm{y}\in\bm{B}^N$, the Cauchy--Schwarz inequality yields
\begin{align}
    \abs*{H^{(p)}_{N}(\bm{y};J^{(p)}-{J'}^{(p)})}
    &\leq 
    \frac{1}{N^{(p-1)/2}} 
    \left(
    \sum_{i_1,\ldots,i_p=1}^N (J_{i_1,\ldots,i_p}^{(p)}-{J'}_{i_1,\ldots,i_p}^{(p)})^2
    \right)^{1/2}
    \left(
    \sum_{i_1,\ldots,i_p=1}^N y_{i_1}^2\cdots y_{i_p}^2
    \right)^{1/2} \nonumber \\
    &\leq
    \sqrt{N} 
    \left(
    \sum_{i_1,\ldots,i_p=1}^N (J_{i_1,\ldots,i_p}^{(p)}-{J'}_{i_1,\ldots,i_p}^{(p)})^2
    \right)^{1/2}
	\nonumber \\    
    &=
    \sqrt{N} 
    \norm*{J^{(p)}-{J'}^{(p)}}_{\R^{N^p}},
    \label{eq:H CS bound}
\end{align}
where $\norm{\cdot}_{\R^{N^p}}$ is the Euclidean norm in $\R^{N^p}$. 
Combining \eqref{eq:sup |H|} and \eqref{eq:H CS bound}, we conclude that
\begin{align*}
    \bigg|
    \sup_{\bm{y}\in\bm{B}^N} 
    \abs*{
    H^{(p)}_{N}(\bm{y};J^{(p)})
    }- 
    \sup_{\bm{y}\in\bm{B}^N} 
    \abs*{
    H^{(p)}_{N}(\bm{y};{J'}^{(p)})
    }
    \bigg|
    \leq 
    \sqrt{N}
    \norm*{J^{(p)}-{J'}^{(p)}}_{\R^{N^p}}.
\end{align*}
Applying \eqref{eq:sup estimate.1} with $\xi(x)=x^p$ yields
\begin{align*}
    \EX{\sup_{\bm{x}\in\bm{B}^{N}}\abs*{H_{N}^{(p)}(\bm{x})}}
    \leq C\sqrt{K\log_+(p)}N\leq C\sqrt{Kp}N,
\end{align*}
and then by Condition~\hyperlink{ass:LS}{(LS)} and  and Proposition 2.3 in \cite{ledoux_concentration_1999}, we conclude that 
\begin{align*}
    &\PR{
    \frac{1}{N}\sup_{\bm{y}\in \bm{B}^N}
    \abs*{H^{(p)}_{N}(\bm{y})}
    \geq C\sqrt{K p} + \sqrt{K} t
    }
    \nonumber \\ 
    &\leq 
    \PR{
    \sup_{\bm{y}\in \bm{B}^N} \abs*{H^{(p)}_{N}(\bm{y})}
    -
    \EX{\sup_{\bm{y}\in \bm{B}^N} \abs*{H^{(p)}_{N}(\bm{y})}}
    \geq \sqrt{K} tN
    } \\
    &\leq \exp(-\frac{t^2}{2}N),
\end{align*}
where $C$ is the same constant as in \eqref{eq:sup estimate.1}.
\end{proof}

\subsection{Proof of Proposition~\ref{prop:Lipschitz}} \label{sec:proof of Lipschitz}

Denote by $\mathfrak{S}_p$ the permutation group of degree $p$. In the proof, we define 
\begin{align}
    J^{(p),\Sym}_{i_1,\ldots,i_p} 
    = \frac{1}{p!}
    \sum_{\tau\in\mathfrak{S}_p} J_{i_{\tau(1)},\ldots,i_{\tau(p)}}^{(p)}
\end{align}
the symmetrized version of $J^{(p)}$, and  define the injective norm 
\begin{align}
    \norm*{J^{(p),\Sym}}_\infty
    =
        \sup_{
        i\in\llbracket 1,p\rrbracket:
        \norm*{\bm{u^i}}=1
        }
        \sum_{i_1,\ldots,i_p=1}^N
        J^{(p),\Sym}_{i_1,\ldots,i_p}
        u^1_{i_1}\cdots u^p_{i_p}
        = 
        \sup_{
        \norm*{\bm{u}}=1
        }
        \sum_{i_1,\ldots,i_p=1}^N
        J^{(p),\Sym}_{i_1,\ldots,i_p}
        u_{i_1}\cdots u_{i_p}
        , \label{eq:tensor norm}
\end{align}
where the second equality follows by Theorem 1~in~\cite{waterhouse_absolute-value_1990}.\footnote{This fact has been rediscovered many times. See Section 2 of \cite{waterhouse_absolute-value_1990} for a brief history.} 
     Since we can write
\begin{align*}
    H_{N}^{(p)}(\bm{x})
    =
    \frac{1}{N^{(p-1)/2}}
    \sum_{i_1,\ldots,i_p=1}^N
    J^{(p),\Sym}_{i_1,\ldots,i_p}
    x_{i_1}\cdots x_{i_p},
\end{align*}
we have
    \begin{align}
        \norm*{J^{(p),\Sym}}_\infty
        =
        \frac{1}{\sqrt{N}} \sup_{\bm{x}\in\bm{B}^N}\abs*{H_{N}^{(p)}(\bm{x})}.
        \label{eq:sym norm}
    \end{align}

\paragraph{Proof of \eqref{eq:p directional}.} 
Fix $i\in\{1,2,3\}$, $p\geq \max\{2,i\}$ and $N\in\N$. For all $\bm{x}\in\bm{B}^N$ and $\bm{v}\in\R^N$ with $\norm{\bm{v}}=1$, we have
\begin{align}
    \partial_{\bm{v}}^i H_{N}^{(p)}(\bm{x})
    =
    \frac{1}{N^{(p-1)/2}} 
    \frac{p!}{(p-i)!}
    \sum_{i_1,\ldots,i_p=1}^N
    J^{(p),\Sym}_{i_1,\ldots,i_p}
    x_{i_1}\cdots x_{i_{p-i}}v_{i_{p-i+1}}\cdots v_{i_p}. \label{eq:tensor expression}
\end{align}
Then, \eqref{eq:tensor norm}, \eqref{eq:tensor expression} and the Cauchy--Schwarz inequality yield
\begin{align}
    \abs*{\partial_{\bm{v}}^i H_{N}^{(p)}(\bm{x})}
    \leq
    \frac{1}{N^{(p-1)/2}}
    \frac{p!}{(p-i)!}
    \norm*{J^{(p),\Sym}}_\infty
    \norm*{\bm{x}}^{p-i}
    \norm*{\bm{v}}^i
    \leq 
    \frac{1}{N^{(i-1)/2}}
    \frac{p!}{(p-i)!}
    \norm*{J^{(p),\Sym}}_\infty. \label{eq:d_v upp bound}
\end{align}
 Combining \eqref{eq:sym norm} and \eqref{eq:d_v upp bound},
 \eqref{eq:sup estimate.2}~of~Lemma~\ref{lem:sup estimate} 
 yields \eqref{eq:p directional}
with the same absolute constant 
$C>0$ as in Lemma~\ref{lem:sup estimate}. 

 \paragraph{Proof of \eqref{eq:p Hessian Op}.} 
 Fix $p\geq 2$ and $N\in\N$. For any $\bm{x}\in\bm{B}^N$ and for any $\bm{u}\in\R^N$ with $\norm{\bm{u}}=1$, we have
\begin{align}
    &\langle
    \nabla^2_E H_{N}^{(p)}(\bm{x})\bm{u},\bm{u}
    \rangle
  =
    \sum_{i,j=1}^N
    \partial_i\partial_jH_{N}^{(p)}(\bm{x})u_iu_j 
    \nonumber \\
    &=
    \frac{1}{N^{(p-1)/2}}
    \sum_{i,j=1}^N
    \left(p(p-1)(1-\delta_{i,j})+\frac{p(p-1)}{2}\delta_{i,j})\right)
    \sum_{i_1,\ldots,i_{p-2}=1}^N
    J^{(p),\Sym}_{i_1,\ldots,i_{p-2},i,j}
    x_{i_1}\cdots x_{i_{p-2}}u_iu_j .\nonumber 
\end{align}
Hence,
\begin{align}
    |\langle
    \nabla^2_E H_{N}^{(p)}(\bm{x})\bm{u},\bm{u}
    \rangle|
  &\leq 
    \frac{p(p-1)}{N^{(p-1)/2}}
    \Big|\sum_{i,j=1}^N
    \sum_{i_1,\ldots,i_{p-2}=1}^N
    J^{(p),\Sym}_{i_1,\ldots,i_{p-2},i,j}
    x_{i_1}\cdots x_{i_{p-2}}u_iu_j\Big| \nonumber \\
    &
    \leq
    \frac{p(p-1)}{N^{(p-1)/2}}
    \norm*{J^{(p),\Sym}}_\infty\norm*{\bm{x}}^{p-2}\norm*{\bm{u}}^2 
    \leq
    \frac{p(p-1)}{\sqrt{N}}
    \sup_{\bm{y}\in\bm{B}^N}\abs*{H_{N}^{(p)}(\bm{y})}, \label{eq:Hess op upp bnd}
\end{align}
where \eqref{eq:Hess op upp bnd}  follows from \eqref{eq:sym norm}. 
Then, \eqref{eq:p Hessian Op} follow from \eqref{eq:Hess op upp bnd} and  Lemma~\ref{lem:sup estimate}, with the same absolute constant $C>0$ as in Lemma~\ref{lem:sup estimate}.


\paragraph{Proof of \eqref{eq:p Hessian Lip}.} 
Fix $p\geq 3$ and $N\in\N$. For any $\bm{x},\bm{y}\in\bm{B}^N$ and for any $\bm{u}\in\R^N$ with $\norm{\bm{u}}=1$, we have
\begin{align}
    &\langle
    (\nabla^2_E H_{N}^{(p)}(\bm{x})-\nabla^2_E H_{N}^{(p)}(\bm{y}))\bm{u},\bm{u}
    \rangle \nonumber \\
    &=
    \sum_{i,j=1}^N
    (\partial_{i}\partial_{j} H_{N}^{(p)}(\bm{x})-\partial_{i}\partial_{j} H_{N}^{(p)}(\bm{y}))u_iu_j \nonumber \\
    &
    =
    \frac{p(p-1)}{N^{(p-1)/2}}
    \sum_{i,j=1}^N
    \sum_{i_1,\ldots,i_{p-2}=1}^N
    J^{(p),\Sym}_{i_1,\ldots,i_{p-2},i,j} 
    (x_{i_1}\cdots x_{i_{p-2}} - y_{i_1}\cdots y_{i_{p-2}})
    u_iu_j. \label{eq:Hessian diff}
\end{align}
By the telescoping sum 
\begin{align*}
    x_{i_1}\cdots x_{i_{p-2}} - y_{i_1}\cdots y_{i_{p-2}}
    =\sum_{\ell=1}^{p-2} x_{i_1}\cdots x_{i_{\ell-1}}(x_{i_{\ell}}-y_{i_{\ell}}) y_{i_{\ell+1}}\cdots y_{i_{p-2}},
\end{align*}
we obtain
\begin{align}
    \eqref{eq:Hessian diff}
    &=
    \frac{p(p-1)}{N^{(p-1)/2}}
    \sum_{\ell=1}^{p-2}
    \sum_{i,j=1}^N
    \sum_{i_1,\ldots,i_{p-2}=1}^N
    J^{(p),\Sym}_{i_1,\ldots,i_{p-2},i,j} 
     x_{i_1}\cdots x_{i_{\ell-1}}(x_{i_{\ell}}-y_{i_{\ell}}) y_{i_{\ell+1}}\cdots y_{i_{p-2}}
    u_iu_j
    \nonumber \\
    &=
    \frac{p(p-1)}{N^{(p-1)/2}}
    \sum_{\ell=1}^{p-2}
    \sum_{i_1,\ldots,i_{p}=1}^N
    J^{(p),\Sym}_{i_1,\ldots,i_p} 
     x_{i_1}\cdots x_{i_{\ell-1}}(x_{i_{\ell}}-y_{i_{\ell}}) y_{i_{\ell+1}}\cdots y_{i_{p-2}}
    u_{p-1}u_p 
    \nonumber
\end{align}
and so
\begin{align}
    &\norm{\nabla^2_E H_{N}^{(p)}(\bm{x})-\nabla^2_E H_{N}^{(p)}(\bm{y})}_{op} \nonumber \\
    &\leq 
    \frac{p(p-1)}{N^{(p-1)/2}}
    \sum_{\ell=1}^{p-2}
    \sum_{i_1,\ldots,i_{p}=1}^N
    \norm*{J^{(p),\Sym}}_{\infty}
    \norm{\bm{x}}^{\ell-1}
    \norm{\bm{x}-\bm{y}}
    \norm{\bm{y}}^{p-\ell-2}
    \norm{\bm{u}}^2
    \\
    &\leq \frac{1}{N}p(p-1)(p-2)
    \norm*{J^{(p),\Sym}}_{\infty}
    \norm{\bm{x}-\bm{y}}. \label{eq:Hess Lip upp bnd}
\end{align}
Combining \eqref{eq:sym norm} and \eqref{eq:Hess Lip upp bnd},
 \eqref{eq:sup estimate.2}~of~Lemma~\ref{lem:sup estimate} 
 yields \eqref{eq:p Hessian Lip}
 with the same absolute constant $C>0$ as in Lemma~\ref{lem:sup estimate}. 

\section{Spectral universality for the truncated Hessian}
\label{sec:truncated model}

Fix $p\geq 2$. We define 
\begin{align*}
    H^{\leq p}_{N}(\bm{x}) 
    = \sum_{r=2}^p \gamma_r H^{(r)}_{N}(\bm{x}), \quad \bm{x}\in\bm{B}^N
\end{align*}
to be the Hamiltonian of the mixed spherical glass model with mixture $\xi_p(z) = \sum_{r=2}^p \gamma_r^2 z^r$. For all $\bm{x}\in\bm{B}^N$, we define the Euclidean Hessian and the spherical Hessian of $H^{\leq p}_{N}$ by
\begin{align*}
\nabla_E^2 H^{\leq p}_{N}(\bm{x}) 
=
\sum_{r=2}^p
\gamma_r
\nabla_E^2 H^{(r)}_{N}(\bm{x}),
\qquad
\nabla^2 H^{\leq p}_{N}(\bm{x}) 
=
P^\intercal_{\bm{x}} \nabla_E^2 H^{\leq p}_{N}(\bm{x}) P_{\bm{x}},
\end{align*}
respectively. 
Denoting by $\lambda_{N,p,\bm{x}}^{(1)},\ldots, \lambda_{N,p,\bm{x}}^{(N)}$ the eigenvalues of $\nabla^2 H^{\leq p}_{N}(\bm{x})$,
we define $\mu_{N,p,\bm{x}} = \frac{1}{N}\sum_{i=1}^N\delta_{\lambda_{N,p,\bm{x}}^{(i)}}$ to be the empirical spectral distribution of $\nabla^2 H^{\leq p}_{N}(\bm{x})$ and $\bar{\mu}_{N,p,\bm{x}}$ to be the mean empirical distribution of $\nabla^2 H^{\leq p}_{N}(\bm{x})$. The main result we prove in this section is the following.

\begin{proposition}
\label{prop:asymptotic freeness}
Assume that \hyperlink{ass:M}{(M)} holds. 
Then, for all $p\geq 2$ and $k\in\N$, we have
\begin{align}
    \lim_{N\rightarrow\infty}
    \sup_{\bm{x}\in\bm{B}^N}
    \bigg|
    \frac{1}{\xi_p''(\rho_{\bm{x}})^{k/2}N} 
    \E
    \Big[
    \Tr (\nabla_E^2 H_N^{\leq p}(\bm{x}))^{k}
    \Big]
    -C_k
    \bigg|
    =
    0,
    \label{eq:truncated moments}
\end{align}
where $\rho_{\bm{x}} = \norm{\bm{x}}^2/N$ and 
\begin{align*}
    C_{k}=
    \begin{cases}
        \frac{1}{k/2+1}{k \choose k/2}, & k \  \text{even}, \\
        0, & k \ \text{odd}.
    \end{cases}
\end{align*}
For $k$ even, $C_k$ is the $k/2$-th Catalan number. The convergence \eqref{eq:truncated moments} yields that for all $p\geq 2$ and $f\in\Lip_1(\mathbb{R})$, we have
\begin{align}
    \lim_{N\rightarrow\infty}
    \sup_{\bm{x}\in\bm{B}^N}
    \abs{
    \int_{y\in\mathbb{R}}f(y)\dd{\bar{\mu}_{N,p,\bm{x}}}(y)
    -
    \int_{y\in\mathbb{R}}f(y)\dd\mu_{sc,2\xi_p''(\rho_{\bm{x}})^{1/2}}(y)
    }
    =0,
    \label{eq:uniform mu bar}
\end{align}
where
\begin{align*}
   \dd\mu_{sc,2\xi_p''(\rho_{\bm{x}})^{1/2}}(y) 
   = \frac{1}{2\pi\xi''(\rho_{\bm{x}})}\sqrt{4\xi_p''(\rho_{\bm{x}}) - y}\Ind_{\abs{y}\leq 2\xi_p''(\rho_{\bm{x}})^{1/2}}\dd{y}. 
\end{align*}
\end{proposition}

\subsection{Proof of \eqref{eq:uniform mu bar} in Proposition \ref{prop:asymptotic freeness}} \label{sec:proof of uniform mu bar}

Denoting by $\lambda_{E,N,p,\bm{x}}^{(1)}\leq \cdots\leq \lambda_{E,N,p,\bm{x}}^{(N)}$ the eigenvalues of $\nabla^2_E H^{\leq p}_{N}(\bm{x})$, we first show that it suffices to show \eqref{eq:uniform mu bar} for 
\[\bar{\mu}_{E,N,p,\bm{x}} = \frac{1}{N}\sum_{i=1}^N \delta_{\lambda_{E,N,p,\bm{x}}^{(i)}},\] 
which is the mean empirical spectral distribution of $\nabla^2_E H^{\leq p}_{N}(\bm{x})$.
Let $\lambda_{N,p,\bm{x}}^{(N)} = 0$ be the eigenvalue of $\nabla^2 H^{\leq p}_{N}(\bm{x})$
corresponding to the eigenvector $\bm{x}$ and $\lambda_{N,p,\bm{x}}^{(1)}\leq \cdots\leq \lambda_{N,p,\bm{x}}^{(N-1)}$ be all other eigenvalues of  $\nabla^2 H^{\leq p}_{N}(\bm{x})$.
Then, the Cauchy interlacing theorem implies that 
\begin{align}
    \lambda_{E,N,p,\bm{x}}^{(i)}
    \leq 
    \lambda_{N,p,\bm{x}}^{(i)}
    \leq 
    \lambda_{E,N,p,\bm{x}}^{(i+1)}, 
    \quad i=1,\ldots,N-1. \label{eq:interlacing}
\end{align}
Thus, for any $f\in\Lip_1(\R)$, we have
\begin{align}
    &\abs{
    \int_\R f(y)
    \dd{\bar{\mu}_{E,N,p,\bm{x}}(y)}
    -
    \int_\R f(y)
    \dd{\bar{\mu}_{N,p,\bm{x}}(y)}
    } \nonumber \\
    &\leq
    \frac{1}{N}\sum_{i=1}^N\EX{
    \abs*{
    f(\lambda_{E,N,p,\bm{x}}^{(i)})-f(\lambda_{N,p,\bm{x}}^{(i)})}
    }
    \nonumber \\
    &\leq
    \frac{1}{N}\sum_{i=1}^N\EX{
    \abs*{\lambda_{E,N,p,\bm{x}}^{(i)}-\lambda_{N,p,\bm{x}}^{(i)}}}
    \nonumber \\
    &\leq
    \frac{1}{N}\sum_{i=1}^{N-1}\EX{
    \lambda_{E,N,p,\bm{x}}^{(i+1)}-\lambda_{E,N,p,\bm{x}}^{(i)}} + \frac{1}{N}\EX{\abs*{\lambda_{E,N,p,\bm{x}}^{(N)}}}
    \nonumber \\
    &=
    \frac{1}{N}(\EX{ \lambda_{E,N,p,\bm{x}}^{(N)}}-\EX{ \lambda_{E,N,p,\bm{x}}^{(1)}} + \EX{\abs*{ \lambda_{E,N,p,\bm{x}}^{(N)}}})
    \nonumber \\
    &\leq 
    \frac{3}{N}\EX{\norm*{\nabla_E^2 H^{\leq p}_{N}(\bm{x})}_{op}}.
    \label{eq:Lip f upp}
\end{align}
By Lemma~\ref{lem:sup estimate} and \eqref{eq:Hess op upp bnd}, we have 
\begin{align*}
\sup_{\bm{x}\in\bm{B}^N}\EX{\norm*{\nabla_E^2 H^{\leq p}_{N}(\bm{x})}_{op}}
\leq
\EX{\sup_{\bm{x}\in\bm{B}^N}\norm*{\nabla_E^2 H^{\leq p}_{N}(\bm{x})}_{op}}
\leq p(p-1)\sqrt{Kp}\sqrt{N},
\end{align*}
so combining this with \eqref{eq:Lip f upp} yields that there exists a constant $C(p,K,\xi)>0$ such that
\begin{align*}
	\sup_{\bm{x}\in\bm{B}^N}
\abs{
    \int_\R f(y)
    \dd{\bar{\mu}_{E,N,p,\bm{x}}(y)}
    -
    \int_\R f(y)
    \dd{\bar{\mu}_{N,p,\bm{x}}(y)}
    }
    &\leq 
    \frac{C(p,K,\xi)}{\sqrt{N}}
    \rightarrow 0,
\end{align*}
as $N\rightarrow\infty$. 
It remains to show that \eqref{eq:truncated moments} implies \eqref{eq:uniform mu bar} where $\bar{\mu}_{N,p,\bm{x}}$ is substituted by $\bar{\mu}_{E,N,p,\bm{x}}$. Fix $f\in\Lip_1(\mathbb{R})$. Denote by $\bar{\nu}_{E,N,p,\bm{x}}$ the mean empirical distribution of 
\[\frac{1}{\xi''_p(\rho_{\bm{x}})^{1/2}}\nabla_E^2 H_N^{\leq p}(\bm{x})\]
and by $\mu_{sc} = \mu_{sc,2}$ the standard semicircular law. Then,
\begin{align}
    &
    \sup_{\bm{x}\in\bm{B}^N}
    \abs{
    \int_{y\in\mathbb{R}}f(y)\dd{\bar{\mu}_{E,N,p,\bm{x}}}(y)
    -
    \int_{y\in\mathbb{R}}f(y)\dd\mu_{sc,2\xi_p''(\rho_{\bm{x}})^{1/2}}(y)
    } 
    \nonumber \\
    &=
    \sup_{\bm{x}\in\bm{B}^N}
    \xi_p''(\rho_{\bm{x}})^{1/2}
    \abs{
    \int_{u\in\mathbb{R}}\tilde{f}(u)\dd{\bar{\nu}_{E,N,p,\bm{x}}}(u)
    -
    \int_{u\in\mathbb{R}}\tilde{f}(u)\dd\mu_{sc}(u)
    }, \label{eq:gDiff}
\end{align}
where $\tilde{f}(u) = \frac{1}{\xi''_p(\rho_{\bm{x}})^{1/2}}f(\xi''_p(\rho_{\bm{x}})^{1/2} u)$. Note that $\tilde{f}\in\Lip_1(\mathbb{R})$ and $\xi_p''(q)^{1/2}\leq \xi_p''(1)^{1/2}$ for any $q\in (0,1]$. Thus,
\begin{align}
    \eqref{eq:gDiff}
    \leq 
    \xi_p''(1)^{1/2}
    \sup_{\bm{x}\in\bm{B}^N}
    \abs{
    \int_{u\in\mathbb{R}}\tilde{f}(u)\dd{\bar{\nu}_{E,N,p,\bm{x}}}(u)
    -
    \int_{u\in\mathbb{R}}\tilde{f}(u)\dd\mu_{sc}(u)
    }.
    \label{eq:gDiff.1}
\end{align}
We claim that for all $\varepsilon>0$ and $g\in\Lip_1(\mathbb{R})$,
\begin{align}
    \sup_{\bm{x}\in\bm{B}^N}
    \abs{
    \int_{u\in\mathbb{R}}g(u)\dd{\bar{\nu}_{E,N,p,\bm{x}}}(u)
    -
    \int_{u\in\mathbb{R}}g(u)\dd\mu_{sc}(u)
    }
    <
    \varepsilon.
    \label{eq:gDiff.2}
\end{align}
Note that combining \eqref{eq:gDiff.2} with \eqref{eq:gDiff.1} yields \eqref{eq:uniform mu bar}. 

We now prove \eqref{eq:gDiff.2}. Fix $\varepsilon>0$. Take $M=\max\{1/\varepsilon,10\}$. By the Stone--Weierstrass theorem, there exists a polynomial $\mathsf{p}$ of degree $L$ such that 
\begin{align}
    \sup_{y\in [-M,M]} \abs{g(y) - \mathsf{p}(y)} < \varepsilon.
    \label{eq:Stone-Weierstrass}
\end{align}
The triangle inequality yields
\begin{align}
    &\sup_{\bm{x}\in\bm{B}^N}
    \abs{
    \int_{u\in\mathbb{R}}g(u)\dd{\bar{\nu}_{E,N,p,\bm{x}}}(u)
    -
    \int_{u\in\mathbb{R}}g(u)\dd\mu_{sc}(u)
    }
    \nonumber \\
    &\leq
    \sup_{\bm{x}\in\bm{B}^N}
    \int_{u\in\mathbb{R}}\abs{g(u)-\mathsf{p}(u)}\dd{\mu_{sc}}(u)
    +
    \sup_{\bm{x}\in\bm{B}^N}
    \int_{u\in\mathbb{R}}\abs{g(u)-\mathsf{p}(u)}\dd{\bar{\nu}_{E,N,p,\bm{x}}}(u)
    \nonumber \\
    &+
    \sup_{\bm{x}\in\bm{B}^N}
    \abs{
    \int_{u\in\mathbb{R}}g(u)\dd{\bar{\nu}_{E,N,p,\bm{x}}}(u)
    -
    \int_{u\in\mathbb{R}}g(u)\dd\mu_{sc}(u)
    }
    \nonumber \\
    &=
    \text{\RomanNumeralCaps 1}
    +
    \text{\RomanNumeralCaps 2}
    +
    \text{\RomanNumeralCaps 3}.
    \label{eq:Lipschitz 3}
\end{align}
By \eqref{eq:Stone-Weierstrass}, \RomanNumeralCaps 1 in \eqref{eq:Lipschitz 3} is bounded by $\varepsilon$. By \eqref{eq:truncated moments}, \RomanNumeralCaps 3 in \eqref{eq:Lipschitz 3} tends to $0$ as $N\rightarrow\infty$. By \eqref{eq:Stone-Weierstrass} and \eqref{eq:truncated moments}, 
\begin{align}
    \limsup_{N\rightarrow\infty}
    \ 
    (\text{\text{\RomanNumeralCaps 2} in \eqref{eq:Lipschitz 3}})
    &\leq 
    \varepsilon
    +
    \limsup_{N\rightarrow\infty}
    \sup_{\bm{x}\in\bm{B}^N}
    \int_{u\in [M,M]^c}\abs{g(u)-\mathsf{p}(u)}\dd{\bar{\nu}_{E,N,p,\bm{x}}}(u). \label{eq:III}
\end{align}
It remains to show that the second term in \eqref{eq:III} is small. Since $g$ is bounded, there exists a constant $C=C_{\mathsf{p},g}$ depending only on $g$ and $\mathsf{p}$ such that for all positive integer $L'$,
\begin{align*}
    &\sup_{\bm{x}\in\bm{B}^N}
    \int_{u\in [M,M]^c}\abs{g(u)-\mathsf{p}(u)}\dd{\bar{\nu}_{E,N,p,\bm{x}}}(u)
    \nonumber \\
    &\leq 
    C_{\mathsf{p},g}
    \sup_{\bm{x}\in\bm{B}^N}
    \int_{u\in [M,M]^c}
    \abs{u}^L 
    \dd{\bar{\nu}_{E,N,p,\bm{x}}}(u)
    \nonumber \\
    &\leq
    C_{\mathsf{p},g}
    M^{-L-2L'}
    \sup_{\bm{x}\in\bm{B}^N}
    \int_{u\in \R}
    u^{2L+2L'} 
    \dd{\bar{\nu}_{E,N,p,\bm{x}}}(u)
\end{align*}
Then by taking the limit supremum, \eqref{eq:truncated moments} yields
\begin{align*}
    \text{second term in \eqref{eq:III}}
    \leq 
    C_{\mathsf{p},g}
    M^{-L-2L'} C_{2L+2L'}
    \leq 
    C_{\mathsf{p},g}
    10^{-L-2L'} 2^{2L+2L'},
\end{align*}
where we apply the facts that $M\geq 10$ and $C_{2\ell}\leq 2^{2\ell}$ to derive the second inequality. Therefore, by taking $L'$ sufficiently large, the second term in \eqref{eq:III} is also bounded from above by $\varepsilon$.
This proves \eqref{eq:gDiff.2}.

\subsection{Proof of \eqref{eq:truncated moments} in  Proposition~\ref{prop:asymptotic freeness}} \label{sec:proof of truncated moments}
Fix $p\geq 2$, $k\in\N$ and $q\in (0,1]$. Without loss of generality, we assume that $k\geq 2$. We use $\tilde J^{\leq p} = \left (\tilde J^{(r)}: r\in\llbracket 2,p\rrbracket\right)$ to denote the disorder in the Gaussian case. Namely, we assume that 
$\tilde J^{(r)}=(\tilde J^{(r)}_{i_1,\ldots,i_r})$ are i.i.d. standard Gaussian random variables. Define the corresponding Gaussian Hamiltonian  
\begin{align*}
    \tilde{H}^{(r)}_{N}(\bm{x})
    =
    H^{(r)}_{N}(\bm{x};\tilde{J}^{(r)}), \quad \bm{x}\in \bm{B}^N.
\end{align*} 
For the Gaussian Hamiltonian, as explained in the fourth paragraph in Section of \cite{following21}, with the zero eigenvalue removed, the eigenvalues of $\nabla^2 \tilde{H}_N^{\leq p}(\bm{x})$ have the same joint distribution as those of
\begin{align}\label{eq:GaussianHessian}
   \sqrt{\frac{N-1}{N}}\xi_p''(\rho_{\bm{x}})^{1/2}\bm{G},
\end{align}
where $\bm{G}$ is a $(N-1)$-dimensional GOE matrix. Hence, Lemma~2.1.6~in~\cite{AGZ_RMT} implies that \eqref{eq:truncated moments} holds for $\nabla^2 \tilde{H}_N^{\leq p}(\bm{x})$.
To prove \eqref{eq:truncated moments}  with general $J^{\leq p}$, we will show that
\begin{align*}
	\lim_{N\rightarrow\infty}
	\sup_{\bm{x}\in\bm{B}^N}
     \frac{1}{\xi''(\rho_{\bm{x}})^{k/2}N} 
     \sup_{\bm{x}\in\bm{B}^N}    
     \abs{
    \EX{\Tr (\nabla^2 H_N^{\leq p}(\bm{x}))^k}
    -
    \EX{\Tr (\nabla^2 \tilde{H}_N^{\leq p}(\bm{x}))^k}}
	=  
    0.
\end{align*}
By the triangle inequality, the left-hand side above is bounded by 
\begin{align}
    \text{\RomanNumeralCaps 1} + \text{\RomanNumeralCaps 2}
    &\coloneqq 
    \sup_{\bm{x}\in\bm{B}^N}
    \frac{1}{\xi''(\rho_{\bm{x}}^{k/2})N}
    \abs{
    \EX{\Tr (\nabla_E^2 \tilde{H}_N^{\leq p}(\bm{x}))^k}
    -
    \EX{\Tr (\nabla^2 \tilde{H}_N^{\leq p}(\bm{x}))^k}
    }
    \nonumber \\   
    &+
	\sup_{\bm{x}\in\bm{B}^N}
     \frac{1}{\xi''(\rho_{\bm{x}})^{k/2}N} 
    \abs{
    \EX{\Tr (\nabla_E^2 H_N^{\leq p}(\bm{x}))^k}
    -
    \EX{\Tr (\nabla_E^2 \tilde{H}_N^{\leq p}(\bm{x}))^k}
    }.
    \label{eq:moment convergence}
\end{align} 

We first argue that Term~\RomanNumeralCaps{1} in \eqref{eq:moment convergence} converges uniformly to $0$. 
Now, since for all $r\in\llbracket 1,p\rrbracket$, the law of $(\tilde H_{r}(\bm{x}))$ is invariant under isometry of $\R^N$ acting on $\bm{x}$, we can replace $\bm{x}$ in Term~\RomanNumeralCaps{1} by the north pole $\bm{n}_{\bm{x}} = (0,\ldots,0,\norm{\bm{x}})$. This yields
\begin{align}
\text{
Term~\RomanNumeralCaps{1} in \eqref{eq:moment convergence}
}
\leq
\frac{1}{N}
\sum_{\bm{j}\in\llbracket 1,N\rrbracket^k}
\abs{T_{\bm{j}}},
\label{eq:trace}
\end{align}
where
\begin{align*}
    T_{\bm{j}} 
    = \EX{(\nabla_E^2 \tilde{H}_{N}^{(p_1)}(\bm{n}_{\bm{x}}))_{j_1,j_2} \cdots (\nabla_E^2 \tilde{H}_{N}^{(p_{k-1})}(\bm{n}_{\bm{x}}))_{j_{k-1},j_k}(\nabla_E^2 \tilde{H}_{N}^{(p_k)}(\bm{n}_{\bm{x}}))_{j_k,j_{k+1}}}
\end{align*}
with the convention that $j_{k+1}=j_1$. 
Define an graph $G_{\bm{j}} = (V_{\bm{j}},E_{\bm{j}})$ where the vertex set $V_{\bm{j}}$ consists of the distinct tuples of $\bm{j}$ and the edge set (permitting loops) is defined as
\begin{align*}
    E_{\bm{j}} = \{\{j_1,j_2\},\ldots,\{j_{k-1},j_k\},\{j_k,j_{k+1}\}\}.
\end{align*}
Define $c:E_{\bm{j}} \rightarrow \llbracket 1,p\rrbracket$ by $c(\{j_\ell,j_{\ell+1}\}) = p_\ell$ for all $\ell\in\llbracket 1,k\rrbracket$.
We also define $\tilde{G}_{\bm{j}}=(V_{\bm{j}},\tilde{E}_{\bm{j}})$ to be the skeleton of $\tilde{G}_{\bm{j}}$ where $\tilde{E}_{\bm{j}}$ is the set of edges in $E_{\bm{j}}$ without multiplicities. 
By the independence and centering of the entries, we see that $T_{\bm{j}}=0$ unless each edge $e\in E_{\bm{j}}$ is of multiplicity at least $2$ and for each $e,e'\in E_{\bm{j}}$ with $e=e'$, $c(e)=c(e')$.
In particular, this implies that $\tilde{E_{\bm{j}}}\leq \floor{k/2}$. Also, note that $\tilde{G}_{\bm{j}}$ is connected, so we have $\abs{V_{\bm{j}}}\leq \floor{k/2}+1$. 
Moreover, for non-vanishing $T_{\bm{j}}$, each edge in $E_{\bm{j}}$ is of multiplicity at least $2$. For each $r\in\llbracket 1,k\rrbracket$, we have 
\begin{align}
    \EX{
    (
    \partial_{j_1}\partial_{j_2}\tilde{H}_{N}^{(p_r)}(\bm{n}_{\bm{x}})
	)^2    
    } 
    \leq 10 p_r^4
    \label{eq:H(n) covariance}
\end{align}
By the H\"{o}lder inequality, \eqref{eq:H(n) covariance} and the fact that $\norm{\cdot}_{s} \leq \norm{\cdot}$ for all $s\geq 2$, \eqref{eq:H(n) covariance} yields that there exists 
a constant $C(p,k)>0$ depending only on $p$ and $k$ such that
\begin{align*}
    \abs{T_{\bm{j}}} \leq \frac{C(p,k)}{N^{k/2}}.
\end{align*}
Therefore,
\begin{align}
	\eqref{eq:trace}
	\leq
    &\frac{C(p,k)}{N^{1+k/2}}\sum_{t=1}^{\floor{k/2}+1}
    \sum_{\bm{j}\in\llbracket 1,N\rrbracket ^k} \Ind
    \{\abs{V_{\bm{j}}} = t\}\Ind\{\exists i\in\llbracket 1,k\rrbracket:j_i=N\} \nonumber \\
    &= 
    \frac{C(p,k)}{N^{1+k/2}}
    \left(
    1+\sum_{t=1}^{\floor{k/2}+1} (N-1)(N-2)\cdots(N-t+1)
    \right) \nonumber \\
    &\leq 
    \frac{C(p,k)}{N}
    \nonumber 
\end{align}
for $N$ sufficiently large.

It remains to consider Term~\RomanNumeralCaps{2} in  \eqref{eq:moment convergence}. By the triangle inequality,
\begin{align}
    &\abs{
    \frac{1}{N}
    \EX{\Tr (\nabla_E^2 H_N^{\leq p}(\bm{x}))^k}
    -
    \frac{1}{N}
    \EX{\Tr (\nabla_E^2 \tilde{H}_N^{\leq p}(\bm{x}))^k}
    } \leq
    \sum_{p_1,\ldots,p_k=2}^p
    \abs*{
    \gamma_{p_1}\cdots\gamma_{p_k}
    } \nonumber \\
    &\quad\quad
    \cdot \abs{
\frac{1}{N} \EX{\Tr \nabla_E^2 H^{(p_1)}_{N}(\bm{x})\cdots\nabla_E^2 H^{(p_k)}_{N}(\bm{x})}
-
\frac{1}{N} \EX{\Tr \nabla_E^2 \tilde{H}^{(p_1)}_{N}(\bm{x})\cdots\nabla_E^2 \tilde{H}^{(p_k)}_{N}(\bm{x})} 
}. 
\label{eq:expansion of trace}
\end{align}
We claim that there exists a constant $C_1(p,k,K)>0$ such that for all $\bm{p}=(p_1,\ldots,p_k)\in\llbracket 2,p\rrbracket^k$ and $\bm{x}\in\bm{B}^N$,
\begin{align}
&D(\bm{x},\bm{p})
\nonumber \\
&\coloneqq
\abs{
\frac{1}{N} \EX{\Tr \nabla_E^2 H^{(p_1)}_{N}(\bm{x})\cdots\nabla_E^2 H^{(p_k)}_{N}(\bm{x})}
-
\frac{1}{N} \EX{\Tr \nabla_E^2 \tilde{H}^{(p_1)}_{N}(\bm{x})\cdots\nabla_E^2 \tilde{H}^{(p_k)}_{N}(\bm{x})} 
}
\nonumber \\
&\leq 
\frac{C_1(p,k,K)}{N}. \label{eq:H and G are close}
\end{align}
Applying \eqref{eq:H and G are close} to the difference in \eqref{eq:expansion of trace} then yields
 the required convergence to zero and completes the proof of \eqref{eq:truncated moments}. Indeed,
the number of ways to choose $\bm{p}$ is finite and independent of $N$. 
Thus, we proceed in the next section to the proof of \eqref{eq:H and G are close}. 

\subsection{Proof of \texorpdfstring{\eqref{eq:H and G are close}}{eq:H and G are close}} \label{sec:proof for H and G}

Throughout this section, we fix $\bm{p}=(p_1,\ldots,p_k)\in\llbracket 2,p\rrbracket^k$ and $\bm{x}\in\bm{B}^N$, and we define  $\abs{\bm{p}}=p_1+\cdots+p_k$.

For the sake of traceability, we break the estimate of $D(\bm{x_N},p)$ in \eqref{eq:H and G are close} into several lemmas.

\begin{lemma}
    \label{lem:first reduction}
    There exists a constant $C_2(p,k,K)>0$ such that $D(\bm{x_N},p)$ defined in \eqref{eq:H and G are close} is bounded from above by 
    \begin{align*}
        D(\bm{x_N},p)
        \leq 
        C_2(p,k,K)
        \sum_{
        \substack{
        \cP:
        \text{partition on $\llbracket 1,k\rrbracket$},
        \\
        \min_{P\in\cP}|P|\geq 2,\,
        \max_{P\in\cP}|P|>2
        } 
        }
        S(\cP)
    \end{align*}
    where $S(\cP)=S_N(\cP)$ is defined by
    \begin{align*}
        S(\cP) = \frac{1}{N^{1+(\abs{\bm{p}}-k)/2}}
        \sum_{j_1,\ldots,j_k=1}^N 
        \sum_{\bm{i}\in\cA(\cP)}
        \prod_{\ell=1}^k 
        |
        \partial_{j_{\ell}}
        \partial_{j_{\ell+1}}
        (x_{i_1^\ell}\cdots x_{i_{p_\ell}^\ell}) 
        |
    \end{align*}
    with the conventions that $j_{k+1}=j_1$,
    that the index $\bm{i}=(\bm{i}^1,\ldots,\bm{i}^k)$, where $\bm{i}^\ell=(i_1^\ell,\ldots,i_{p_\ell}^\ell)$, 
runs over $\llbracket 1,N\rrbracket^{\bm{p}} \coloneqq \llbracket 1,N\rrbracket^{p_1}\times\cdots\times \llbracket 1,N\rrbracket^{p_k}$, and that
    \begin{align*}
    \cA(\cP)
    =
    \left
    \{ \bm{i}\in\llbracket 1,N\rrbracket^{\bm{p}}:
    \text{
    for all $\ell_1,\ell_2\in P\in\cP$, $p_{\ell_1}=p_{\ell_2}$ and $\bm{i}^{\ell_1} = \bm{i}^{\ell_2}$
    }
    \right
    \}.
\end{align*}
\end{lemma}
\begin{proof}[Proof of Lemma~\ref{lem:first reduction}]
    Note that
\begin{align}
&\frac{1}{N}\EX{\Tr \nabla_E^2 H^{(p_1)}_{N}(\bm{x})\cdots\nabla_E^2 H^{(p_k)}_{N}(\bm{x})} \nonumber \\ 
&= 
\frac{1}{N^{1+(\abs{\bm{p}}-k)/2}}
\sum_{j_1,\ldots,j_k=1}^N 
\sum_{\bm{i}\in\llbracket 1,N\rrbracket^{\bm{p}}}
\Big( \EX{
J_{i_1^1,\ldots,i^1_{p_1}}^{(p_1)}\cdots J_{i_1^k,\ldots,i^k_{p_k}}^{(p_k)}
}
\prod_{\ell=1}^k \partial_{j_{\ell}}\partial_{j_{\ell+1}}(x_{i_1^\ell}\cdots x_{i_{p_\ell}^\ell}) 
\Big).
\label{eq:tr expansion}
\end{align}
Hence, the left-hand side of \eqref{eq:H and G are close} is bounded by
\begin{align}
\frac{1}{N^{1+(\abs{\bm{p}}-k)/2}}
\sum_{j_1,\ldots,j_k=1}^N 
\sum_{\bm{i}\in\llbracket 1,N\rrbracket^{\bm{p}}}
 D_{\bm{i}}
\prod_{\ell=1}^k \big|\partial_{j_{\ell}}\partial_{j_{\ell+1}}(x_{i_1^\ell}\cdots x_{i_{p_\ell}^\ell}) \big|
,
\label{eq:tr diff}
\end{align}
where the index $\bm{i}$ runs over $\llbracket 1,N\rrbracket^{p_1}\times\cdots\times \llbracket 1,N\rrbracket^{p_k}$ and we define 
\begin{align*}
    D_{\bm i}:=  
\big|\E
J_{i_1^1,\ldots,i^1_{p_1}}^{(p_1)}\cdots J_{i_1^k,\ldots,i^k_{p_k}}^{(p_k)} -\E \tilde J_{i_1^1,\ldots,i^1_{p_1}}^{(p_1)}\cdots \tilde J_{i_1^k,\ldots,i^k_{p_k}}^{(p_k)}
\big|
.
\end{align*}
To bound \eqref{eq:tr diff}, for each $\cP$ being a partition on $\llbracket 1,k\rrbracket$,
define the partial sum $\tilde{S}(\cP)=\tilde{S}_N(\cP)$ by
\begin{align}\label{eq:1term}
\tilde{S}(\cP):=\frac{1}{N^{1+(\abs{\bm{p}}-k)/2}}
\sum_{j_1,\ldots,j_k=1}^N 
\sum_{\bm{i}\in \cA(\cP)}
D_{\bm i}
\prod_{\ell=1}^k \big|\partial_{j_{\ell}}\partial_{j_{\ell+1}}(x_{i_1^\ell}\cdots x_{i_{p_\ell}^\ell}) 
\big|,
\end{align}
where $\cA(\cP)$ was introduced in the statement of the lemma. 

Note that for each $\bm{i}\in\llbracket 1,N\rrbracket^{\bm{p}}$ in \eqref{eq:tr diff} there exists  at least one $\cP$ such that $\bm i\in \cA(\cP)$ (more than one in cases that there is $\cP'$ such that $\cP$ is a refinement of $\cP'$). 
Hence,  \eqref{eq:tr diff} is bounded by 
\begin{align}
\sum_{\cP:\text{partition on $\llbracket 1,k\rrbracket$}} \tilde{S}(\cP). \label{eq:sum SP}
\end{align}

Recall that the variables $J^{(m)}_{i_1,\ldots,i_{p_m}}$ are assumed to be centered with variance $1$. Hence, by independence, suppose that $\cP$ is a partition on $\llbracket 1,k\rrbracket$ such that either there is $P\in \cP$ such that $|P|=1$ or $|P|=2$ for all $P\in\cP$. Then, $D_{\bm{i}}=0$ for all $\bm{i}\in\cA(\cP)$ which yields $\tilde{S}(\cP)=0$. Hence, we may assume that the summation in \eqref{eq:sum SP} runs over partitions $\cP$ on $\llbracket 1,k\rrbracket$ such that $\min_{P\in\cP} |P|\geq 2$ and  $\max_{P\in\cP}|P|>2$. 
Moreover, by Condition~\hyperlink{ass:M}{(M)}, $D_{\bm i}\leq C_2(p,k,K)$ for
appropriate $C_2(p,k,K)>0$ which yields
\begin{align*}
    \tilde{S}(\cP) \leq C_2(p,k,K) S(\cP), 
\end{align*}
where $S(\cP)$ was introduced in the statement of the lemma. 
Hence, \eqref{eq:tr diff} is bounded from above by
\begin{align*}
 C_2(p,k,K)
        \sum_{
        \substack{
        \cP:
        \text{partition on $\llbracket 1,k\rrbracket$},
        \\
        \min_{P\in\cP}|P|\geq 2,\,
        \max_{P\in\cP}|P|>2
        } 
        }
        S(\cP),
\end{align*}
and the proof is completed.
\end{proof}

By Lemma~\ref{lem:first reduction} and the fact that the number of partitions $\cP$ is finite and independent of $N$, to prove \eqref{eq:H and G are close} it is enough to show that there exists a constant $C_3(p,k)>0$ such that for any fixed $\cP$ with $\min_{P\in\cP}|P|\geq 2$ and  $\max_{P\in\cP}|P|>2$, 
\begin{equation}
    S(\cP)\leq \frac{C_3(p,k)}{N}.\label{eq:Sdiff}
\end{equation}
Henceforth, we fix the choice of such a partition $\cP$ and prove \eqref{eq:Sdiff}. Fixing $\bm{j}=(j_1,\ldots,j_k)\in \llbracket 1,N\rrbracket^k$ and $\bm{i}\in\cA(\cP)$, we want to estimate the product 
\begin{align*}
    \prod_{\ell=1}^k
    \abs*{
    \partial_{j_\ell}\partial_{j_{\ell+1}}(x_{i_1^\ell}\cdots x_{i_{p_\ell}^\ell})
    }
\end{align*}
appeared in $S(\cP)$. The product rule for derivatives and the triangle inequality yield
\begin{align}
    \prod_{\ell=1}^k
    \abs*{
    \partial_{j_\ell}\partial_{j_{\ell+1}}(x_{i_1^\ell}\cdots x_{i_{p_\ell}^\ell})
    }
    &\leq
    \prod_{\ell=1}^k
    \sum_{
    \substack{
    I(\ell,1),I(\ell,2)\in\llbracket 1,p_\ell\rrbracket :
    \\
    I(\ell,1)\neq I(\ell,2)
    }
    }
    \Ind_{\{j_\ell = i^\ell_{I(\ell,1)},\, j_{\ell+1} = i^\ell_{I(\ell,2)}\}}
    \bigg(
    \prod_{
    \substack{
    r\in\llbracket 1,p_\ell\rrbracket : 
    \\
    r\neq I(\ell,1),I(\ell,2)
    }
    }
    \abs*{
    x_{i_r^\ell}
    }
    \bigg) 
    \nonumber \\
    &\eqqcolon 
    \prod_{\ell=1}^k 
    \sum_{
    \substack{
    I(\ell,1),I(\ell,2)\in\llbracket 1,p_\ell\rrbracket :
    \\
    I(\ell,1)\neq I(\ell,2)
    }
    }
    F(\ell,I(\ell,1),I(\ell,2)).
    \label{eq:F_lI}
\end{align}

Denote by $\mathcal{I}_k$ the collection of mappings $I:\llbracket 1,k\rrbracket\times \{1,2\}\to \llbracket 1,N\rrbracket$ such that 
for all $\ell\in\llbracket 1,k\rrbracket$, $I(\ell,1)\neq I(\ell,2)$ and $I(\ell,1),I(\ell,2)\in\llbracket 1,p_\ell\rrbracket$.
By interchanging the product and the sum in \eqref{eq:F_lI}, we have
\begin{align}
    \eqref{eq:F_lI}
    =
    \sum_{I\in\mathcal{I}_k}
    \prod_{\ell=1}^k
    F(\ell,I(\ell,1),I(\ell,2)). \label{eq:prod F_lI}
\end{align}
Note that the cardinality of $\mathcal{I}_k$ is finite and independent of $N$. Thus, to prove \eqref{eq:Sdiff}, it suffices to show that there exists a constant $C_4(p,k)>0$ such that for any $I\in\mathcal{I}_{k}$,
\begin{align}
    S(\cP,I)
    \coloneqq 
    \frac{1}{N^{1+(\abs{\bm{p}}-k)/2}}
    \sum_{j_1,\ldots,j_k=1}^N 
    \sum_{\bm{i}\in\cA(\cP)}
    \prod_{\ell=1}^k
    F(\ell,I(\ell,1),I(\ell,2))
    \leq 
    \frac{C_4(p,k)}{N}. \label{eq:S_PI}
\end{align}
As such, we proceed with proving \eqref{eq:S_PI} with $I\in\mathcal{I}_k$ fixed. The product in \eqref{eq:prod F_lI} factorizes into
\begin{align}
    \prod_{\ell=1}^k
    F(\ell,I(\ell,1),I(\ell,2))
    =
    \bigg(
    \prod_{\ell=1}^k
    \Ind_{\{j_\ell = i^\ell_{I(\ell,1)},\, j_{\ell+1}
    = i^\ell_{I(\ell,2)}\}}
    \bigg)
    \bigg(
    \prod_{\ell=1}^k
    \prod_{
    \substack{
    r\in\llbracket 1,p_\ell\rrbracket : 
    \\
    r\neq I(\ell,1),I(\ell,2)
    }
    }
    \abs*{
    x_{i_r^\ell}
    }
    \bigg).
    \label{eq:F factors}
\end{align}
We define the indicator set
\begin{align*}
    \mathcal{B}(\bm{j},I)
    =
    \bigcap_{\ell=1}^k
    \mathcal{B}(\bm{j},\ell,I)
    =
    \bigcap_{\ell=1}^k
    \big\{
    \bm{i}\in\mathcal{A}(\cP):
    \text{
    $j_\ell = i^\ell_{I(\ell,1)}$ and $j_{\ell+1}
    = i^\ell_{I(\ell,2)}$
    }
    \big\}.
\end{align*}
Then the first term of \eqref{eq:F factors} equals $\Ind_{\bm{i}\in\mathcal{B}(\bm{j},I)}$. Moreover, introducing 
\begin{align}
    d(r,\ell,I) = 
    \begin{cases}
        1, &  r\neq I(\ell,1),I(\ell,2), \\
        0, &  \text{otherwise},
    \end{cases} \label{eq:degree}
\end{align}
the second term of \eqref{eq:F factors} equals 
\begin{align*}
    \prod_{\ell=1}^k \prod_{r=1}^{p_\ell} \abs*{x_{i^\ell_r}}^{d(r,\ell,I)}.
\end{align*}
Therefore, plugging the expressions back to \eqref{eq:F factors} and then applying \eqref{eq:F factors} yield
\begin{align*}
    S(\cP,I)
    =
    \frac{1}{N^{1+(\abs{\bm{p}}-k)/2}}
    \sum_{j_1,\ldots,j_k=1}^N 
    \sum_{\bm{i}\in\mathcal{B}(\bm{j},I)}
    \prod_{\ell=1}^k \prod_{r=1}^{p_\ell} \abs*{x_{i^\ell_r}}^{d(r,\ell,I)} 
    \eqqcolon
    \frac{1}{N^{1+(\abs{\bm{p}}-k)/2}}
    \sum_{\bm{j}\in\llbracket 1, N\rrbracket^k} T(\bm{j},\cP,I).
\end{align*}
To prove \eqref{eq:S_PI}, we claim that for any $I\in\mathcal{I}_k$, 
\begin{align}
    \sum_{\bm{j}\in\llbracket 1, N\rrbracket^k} T(\bm{j},\cP,I)
    \leq N^{(\abs*{\bm{p}}-k)/2}.
    \label{eq:T upp bnd}
\end{align}
The proof of \eqref{eq:T upp bnd} is provided in Section~\ref{sec:proof of T upp bnd}, and it is immediate that \eqref{eq:T upp bnd} yields \eqref{eq:S_PI}.

\section{Proof of \eqref{eq:T upp bnd}} \label{sec:proof of T upp bnd}

This section aims to prove \eqref{eq:T upp bnd}. The following lemmas do this. Before we state the lemmas, we recall that $\cP$ is a partition on $\llbracket 1,k\rrbracket$ such that $\min_{P\in\cP} |P|\geq 2$ and  $\max_{P\in\cP}|P|>2$, and $I\in\mathcal{I}_k$ where $\mathcal{I}_k$ consists of $I:\llbracket 1,k\rrbracket\times \{1,2\}\to \llbracket 1,N\rrbracket$ such that 
for all $\ell\in\llbracket 1,k\rrbracket$, $I(\ell,1)\neq I(\ell,2)$ and $I(\ell,1),I(\ell,2)\in\llbracket 1,p_\ell\rrbracket$.

\begin{lemma}
    \label{lem:T upp bnd}
    We have
    \begin{align*}
        \sum_{\bm{j}\in\llbracket 1, N\rrbracket^k} T(\bm{j},\cP,I)
        \leq 
        N^{\frac{1}{2}\abs*{\bm{p}} - \frac{1}{2}\sum_{P\in\cP}\abs{P}\abs*{I(P\times\{1,2\})}}
        \prod_{Q\in\cQ_{\cP,I}}
        \sum_{j_Q\in \llbracket 1,N\rrbracket} \abs*{x_{j_Q}}^{d_Q}.
    \end{align*}
    Here, $\cQ_{\cP,I}$ is the partition on $\llbracket 1,k\rrbracket\times \{1,2\}$ defined by the relation $\mathfrak{R}_1\cup \mathfrak{R}_2$ where
    \begin{align*}
        &\mathfrak{R}_1 
        \coloneqq \bigcup_{\ell=1}^k\{((\ell+1,1),(\ell,2))\},
        \\
        &\mathfrak{R}_2
        \coloneqq
        \bigcup_{P\in\cP} 
        \Big\{
        ((\ell_1,i_1),(\ell_2,i_2)) \in (P\times\{1,2\})^2
        :
        I(\ell_1,i_1) = I(\ell_2,i_2) 
        \Big\},
    \end{align*}
    where we identify $(k+1,1)$ as $(1,1)$.
    Also, for all $Q\in \cQ_{\cP,I}$,
    \begin{align*}
        d_Q
        \coloneqq
        \sum_{P\in\cP_Q}\sum_{r\in I(Q\cap (P\times\{1,2\}))}\abs*{P} - \abs*{I^{-1}(r)\cap (P\times\{1,2\})}
    \end{align*}
    where 
    \begin{align*}
        \cP_Q = 
        \{P\in\cP : 
        Q\cap (P\times \{1,2\}) \neq \varnothing
        \}.
    \end{align*}
\end{lemma}

\begin{lemma}
\label{lem:increment of exponent}
    Define $\alpha(\cP,I)$ to be the smallest number such that for every $(x_1,\ldots,x_N)\in\bm{B}^N$,
\begin{align*}
    \prod_{Q\in \cQ_{\cP,I}}\sum_{j_Q=1}^N \abs*{x_{j_Q}}^{d_{Q}} \leq N^{-k + \frac{1}{2}\sum_{P\in\cP}\abs*{P}\abs*{I(P\times\{1,2\})} +\alpha(\cP,I)}. 
\end{align*}
Then, we have $\alpha(\cP,I)\leq \frac{k}{2}$.
\end{lemma}

The proofs of Lemma~\ref{lem:T upp bnd} and Lemma~\ref{lem:increment of exponent} are provided in Section~\ref{sec:proof of lem:T upp bnd} and Section~\ref{sec:increment of exponent}, respectively. We now show that Lemma~\ref{lem:T upp bnd} and Lemma~\ref{lem:increment of exponent} imply \eqref{eq:T upp bnd}.

\begin{proof}[Proof of \eqref{eq:T upp bnd}]
    By Lemma~\ref{lem:T upp bnd} and Lemma~\ref{lem:increment of exponent}, we have
    \begin{align*}
        \sum_{\bm{j}\in\llbracket 1, N\rrbracket^k} T(\bm{j},\cP,I)
        \leq 
        N^{(\abs*{\bm{p}}-k)/2},
    \end{align*}
    which completes the proof.
\end{proof}

\subsection{Proof of Lemma~\ref{lem:T upp bnd}}
\label{sec:proof of lem:T upp bnd}

Fixing $\bm{j}\in \llbracket 1, N\rrbracket^k$, $\bm{i}\in \mathcal{B}(\bm{j},I)$ and $P\in\cP$, we can identify $p_\ell = p_P$ and $\bm{i}^\ell = \bm{i}^P$ for all $\ell\in P$. Thus, by interchanging the products $\prod_{\ell\in P}$ and $\prod_{r=1}^{p_P}$, we obtain 
\begin{align}
    \prod_{\ell=1}^k 
    \prod_{r=1}^{p_\ell} 
    \abs*{x_{i^\ell_r}}^{d(r,\ell,I)}
    =
    \prod_{P\in\cP}
    \prod_{\ell\in P}
    \prod_{r=1}^{p_P}
    \abs*{x_{i^P_r}}^{d(r,\ell,I)}
    =
    \prod_{P\in\cP}
    \prod_{r=1}^{p_P}
    \abs*{x_{i^P_r}}^{\sum_{\ell\in P} d(r,\ell,I)}. \label{eq:prod x}
\end{align}
For $\bm{i}\in\mathcal{B}(\bm{j},I)$, we introduce the decomposition $\bm{i} = (\bm{i}_1,\bm{i}_2)$ where $\bm{i}_1 = (i^P_r)_{r\in I(P\times\{1,2\}), P\in\cP}$ and $\bm{i}_2 = (i^P_r)_{r\not\in I(P\times\{1,2\}),P\in\cP}$. Then by \eqref{eq:prod x}, $T(\bm{j},\cP,I)$ admits the factorization
\begin{align}
    &T(\bm{j},\cP,I)
    \nonumber \\
    &=
    \sum_{(\bm{i}_1,\bm{i}_2)\in \mathcal{B}(\bm{j},I)}
    \bigg(
    \prod_{P\in\cP}
    \prod_{r\in I(P\times \{1,2\})}
    \abs*{x_{i^P_r}}^{\sum_{\ell\in P} d(r,\ell,I)}
    \bigg)
    \bigg(
    \prod_{P\in\cP}
    \prod_{r\notin I(P\times \{1,2\})}
    \abs*{x_{i^P_r}}^{\sum_{\ell\in P} d(r,\ell,I)}
    \bigg) 
    \nonumber \\
    &\eqqcolon
    \sum_{(\bm{i}_1,\bm{i}_2)\in \mathcal{B}(\bm{j},I)}
    T_1(\mathcal{P},I,\bm{i}_1) 
    \cdot
    T_2(\mathcal{P},I,\bm{i}_2). \label{eq:T factors}
\end{align}
Denoting by $I^{-1}(r)_1$ and $I^{-1}(r)_2$ the first and second coordinates of $I^{-1}(r)$, respectively, define
\begin{align*}
        \mathcal{B}_1(\bm{j},\cP,I)
        =
        \Big\{
    \bm{i}_1:
    \text{ 
    $i^P_r\in \llbracket 1,N\rrbracket$, 
    $i^P_r = j_{I^{-1}(r)_1+I^{-1}(r)_2-1}$ 
    for all $r\in I(P\times\{1,2\})$ and $P\in\cP$
    }
    \Big\}.
    \end{align*}
Then, $\bm{i} = (\bm{i}_1,\bm{i_2})\in \mathcal{B}(\bm{j},I)$ if and only if $\bm{i}_1\in \mathcal{B}_1(\bm{j},\cP,I)$. Applying this to \eqref{eq:T factors}, $T(\bm{j},\cP,I)$ can be further factorized into 
\begin{align}
     T(\bm{j},\cP,I)
     =
     \bigg(
     \sum_{\bm{i}_1\in\mathcal{B}_1(\bm{j},\mathcal{P},I)} T_1(\mathcal{P},I,\bm{i}_1)
     \bigg)
     \bigg(
    \prod_{P\in\cP}
    \prod_{r\notin I(P\times \{1,2\})}
     \sum_{\bm{i}_r^P=1}^N \abs*{x_{i^P_r}}^{\sum_{\ell\in P} d(r,\ell,I)}
     \bigg) 
     \eqqcolon
     \text{\RomanNumeralCaps 1} 
        \cdot
        \text{\RomanNumeralCaps 2}
    .\label{eq:factorization}
\end{align}
    We now treat \RomanNumeralCaps 1 and \RomanNumeralCaps 2 in \eqref{eq:factorization} separately. Before proceeding, we note that by the definition of $d(r,\ell,I)$ provided in \eqref{eq:degree} and the assumption that $I(\ell,1)\neq I(\ell,2)$ for all $\ell\in\llbracket 1,k\rrbracket$, we see that 
    \begin{align}
        d(r,\ell,I) = 
        1
        -
        \abs*{I^{-1}(r)\cap \{(\ell,1),(\ell,2)\}}. 
        \label{eq:d(r,l,I)}
    \end{align}
    We now proceed with \RomanNumeralCaps 2. 
    \paragraph{\RomanNumeralCaps 2 in \eqref{eq:factorization}.}
    Since $r\notin I(P\times \{1,2\})$ yields $\abs*{I^{-1}(r)\cap \{(\ell,1),(\ell,2)\}} = 0$ for $\ell\in P$, by \eqref{eq:d(r,l,I)},
    \begin{align*}
        \sum_{\ell\in P} d(r,\ell,I)
        =\abs{P}.
    \end{align*}
    Moreover, by the assumption that $\abs{P}\geq 2$ and the fact that $\norm{\cdot}_{s}\leq \norm{\cdot}_{2}$ for all $s\geq 2$, interchanging the product and the sum of {\RomanNumeralCaps 2} in \eqref{eq:factorization} yields
    \begin{align}
        \text{
        \RomanNumeralCaps 2 in \eqref{eq:factorization}
        }
        =
        \prod_{P\in\cP}
        \prod_{r\not\in I(P\times\{1,2\})}
        \sum_{i^P_r=1}^N
        \abs*{
    x_{i_r^P}
    }^{\abs{P}}
        \leq 
        N^{\frac{1}{2}\sum_{P\in\cP}\abs{P}(p_P - \abs{I(P\times\{1,2\})})}, \nonumber 
    \end{align}
    and because $\sum_{P\in\cP}\abs*{P}p_P = \abs*{\bm{p}}$, the upper bound above yields
    \begin{align}
        \text{
        \RomanNumeralCaps 2 in \eqref{eq:factorization}
        }
        \leq
        N^{\frac{1}{2}\abs*{\bm{p}} - \frac{1}{2}\sum_{P\in\cP}\abs{P}\abs{I(P\times\{1,2\})}}. \label{eq:factorization.2}
    \end{align}

    \paragraph{\RomanNumeralCaps 1 in \eqref{eq:factorization}.}
    Substituting $d(r,\ell,I)$ with \eqref{eq:d(r,l,I)},
    \begin{align}
        \text{\RomanNumeralCaps 1 in \eqref{eq:factorization}}
        &=
        \sum_{\bm{i}_1\in \mathcal{B}_1(\bm{j},\cP,I)}
        \prod_{P\in\cP}
        \prod_{r\in I(P\times\{1,2\})}
        \abs*{
    x_{i_r^P}
    }^{\abs*{P}-\abs*{I^{-1}(r)\cap (P\times\{1,2\})}}.
        \label{eq:factorization.1}
    \end{align}
    Thus, it remains to study the identification between the indices $\bm{i}_1$ and $\bm{j}$. To do so, we introduce $\bm{j}_1 = (j_{1,1},\ldots,j_{k,1})\in\llbracket 1,N\rrbracket^k$ and $\bm{j}_2=(j_{1,2},\ldots,j_{k,2})\in\llbracket 1,N\rrbracket^k$ which are duplications of $\bm{j}$.
    We define the indicator set
    \begin{align*}
    \mathcal{J}
        &= \big\{(\bm{j}_1,\bm{j_2}) : \text{for all $\ell\in\llbracket 1,k\rrbracket$, $j_{\ell+1,1}=j_{\ell,2}$} \big\},
    \end{align*}
    with the convention that $j_{k+1,1}=j_{1,1}$. Then, for all $\bm{j}\in\llbracket 1,N\rrbracket^k$ and $(\bm{j}_1,\bm{j}_2)\in\mathcal{B}_1$ with $\bm{j}=\bm{j}_1$, we have 
    \begin{align*}
        \mathcal{B}_1(\bm{j},\cP,I)
        &=
        \Big\{
        \bm{i}_1:
        \text{ 
        $i^P_r\in \llbracket 1,N\rrbracket$, 
        $i^P_r = j_{I^{-1}(r)_1+I^{-1}(r)_2-1}$ 
        for all $r\in I(P\times\{1,2\})$ and $P\in\cP$
        }
        \Big\}
        \\
        &=
        \Big\{
        \bm{i}_1:
        \text{
        $i^P_r\in \llbracket 1,N\rrbracket$,
        $i^P_r = j_{I^{-1}(r)}$
        for all $r\in I(P\times\{1,2\})$ and $P\in\cP$
        }
        \Big\}
        \\
        &\eqqcolon
        \mathcal{B}_1((\bm{j}_1,\bm{j}_2),\cP,I).
    \end{align*}
    We now state a lemma to identify the indices in $\mathcal{J}$ and $\mathcal{B}_1((\bm{j}_1,\bm{j}_2),\cP,I)$.
    \begin{lemma}
    \label{lem:i and j}
    Let $(\bm{j}_1,\bm{j}_2)\in\mathcal{J}$ and $\bm{i}_1\in \mathcal{B}_1((\bm{j}_1,\bm{j}_2),\cP,I)$. 
    Then, the following statements hold.
        \begin{enumerate}
        \item If $(\ell_1,i_1),(\ell_2,i_2)\in Q\in \cQ_{\cP,I}$, then $j_{\ell_1,i_1}=j_{\ell_2,i_2}$.
        \item For all $P\in\cP$ and $Q\in\cQ_{\cP,I}$, if $r\in I((P\times \{1,2\})\cap Q)$, then $i^P_r = j_{\ell,m}$ for all $(\ell,m)\in Q$.
        \end{enumerate}
    \end{lemma}
    \begin{proof}[Proof of Lemma~\ref{lem:i and j}]
        We start with the proof of the first statement. 
        If suffices to check for the pairs $(\ell_1,i_1),(\ell_2,i_2)\in\mathfrak{R}_1\cup \mathfrak{R}_2$. If $(\ell_1,i_1),(\ell_2,i_2)\in \mathfrak{R}_1$, then the first statement follows from the assumption that $(\bm{j}_1,\bm{j}_2)\in\mathcal{J}$. If $(\ell_1,i_1),(\ell_2,i_2)\in \mathfrak{R}_2$, then we have $\ell_1,\ell_2\in P$ for some $P\in\cP$ and  $I(\ell_1,i_1)=I(\ell_2,i_2)$, and these yield
        \begin{align*}
            j_{\ell_1,i_1} = i^P_{I(\ell_1,i_1)} = i^P_{I(\ell_2,i_2)} = j_{\ell_2,i_2}.
        \end{align*}
        It remains to prove the second statement. Fixing $P\in\cP$ and $Q\in\cQ_{\cP,I}$, for all $r\in I((P\times \{1,2\})\cap Q)$, there exists $(\ell,m)\in (P\times \{1,2\})\cap Q$ such that $r=I(\ell,m)$, so $i^P_r = i^{P}_{I(\ell,m)} = j_{\ell,m}$. Then the second statement follows from the first statement. 
    \end{proof}
    Lemma~\ref{lem:i and j} yields that if $(\bm{j}_1,\bm{j}_2)\in\mathcal{J}$ and $\bm{i}_1\in \mathcal{B}_1((\bm{j}_1,\bm{j}_2),\cP,I)$, then we can identify the indices $j_{\ell,i} = j_Q$, for $(\ell,i)\in Q\in \cQ_{\cP,I}$,  and $i_r^P = j_Q$, for $r\in I((P\times\{1,2\})\cap Q)$ with $P\in \cP$ and $Q\in\cQ_{\cP,I}$. Combining this with \eqref{eq:factorization.2} then yields
    \begin{align}
        \sum_{\bm{j}\in\llbracket 1,N\rrbracket^k}
        (
        \text{\RomanNumeralCaps 1 in \eqref{eq:factorization}}
        )
        &\leq 
        \sum_{j_Q\in\llbracket 1,N\rrbracket:Q\in\cQ_{\cP,I}}
        \prod_{Q\in\cQ_{\cP,I}}
        \prod_{P\in\cP}
        \prod_{r\in I((P\times\{1,2\})\cap Q)}
        \abs*{
    x_{j_Q}
    }^{\abs*{P}-\abs*{I^{-1}(r)\cap (P\times\{1,2\})}} 
    \nonumber \\
        &=
        \prod_{Q\in\cQ_{\cP,I}}
        \sum_{j_Q=1}^N
        \abs*{x_{j_Q}}^{d_Q},
        \label{eq:d_Q bnd}
    \end{align}
    where \eqref{eq:d_Q bnd} is derived from interchanging the sum $\sum_{j_Q\in\llbracket 1,N\rrbracket:Q\in\cQ_{\cP,I}}$ and the product $\prod_{Q\in\cQ_{\cP,I}}$ and then apply the definition of $d_Q$ given in the statement of Lemma~\ref{lem:T upp bnd}.

    Finally, combining \eqref{eq:factorization.2} and \eqref{eq:d_Q bnd} yields the desired upper bound for $\sum_{\bm{j}\in\llbracket 1, N\rrbracket^k} T(\bm{j},\cP,I)$.

\subsection{Proof of Lemma~\ref{lem:increment of exponent}} \label{sec:increment of exponent}
For any $Q\in\cQ_{\cP,I}$, the following lemma implies that $d_Q \geq 0$. 
\begin{lemma}
\label{lem:card}
Let $\cP$ be a partition on $\llbracket 1,k\rrbracket$ and $I\in\mathcal{I}_k$. Then, for all $P\in\cP$ and $r\in\llbracket 1,p_P\rrbracket$,  we have $\abs*{I^{-1}(r)\cap (P\times\{1,2\})}\leq \abs{P}$.
\end{lemma}
\begin{proof}[Proof of Lemma~\ref{lem:card}]
Fix $P\in\cP$ and $r\in\llbracket 1,p_P\rrbracket $. Then, by the assumption that $I(\ell,1)\neq I(\ell,2)$ in the beginning of Section~\ref{sec:proof of T upp bnd}, we have $\abs{I^{-1}(r)\cap \{(\ell,1),(\ell,2)\}}\leq 1$. Thus,
\begin{align*}
    \abs*{I^{-1}(r)}
    = \sum_{\ell\in\cP} \abs*{I^{-1}(r)\cap \{(\ell,1),(\ell,2)\}}
    \leq \abs{P}, 
\end{align*}
and the proof is completed.
\end{proof}
Thanks to Lemma~\ref{lem:card}, 
\begin{align*}
    d_Q = \sum_{P\in\cP_Q}\sum_{r\in I(Q\cap (P\times\{1,2\}))}\abs*{P} - \abs*{I^{-1}(r)\cap (P\times\{1,2\})} \geq 0.
\end{align*}
Thus, there are only three possible cases of $d_Q$.
\begin{enumerate}
    \item If $d_Q\geq 2$, by the fact that $\norm{\cdot}_{d_Q}\leq \norm{\cdot}_{2}$, we have $\sum_{j_Q=1}^N \abs*{x_{j_Q}}^{d_Q}\leq N^{d_Q/2}$. Moreover, the equality is attained by taking $x_1=N$ and $x_i=0$ for all $i=2,\ldots,N$.
    \item If $d_Q=1$, then by the Cauchy--Schwarz inequality we have $\sum_{j_Q=1}^N \abs*{x_{j_Q}}^{d_Q}\leq N$. Moreover, the equality is attained by taking $x_i=1$ for all $i=1,\ldots,N$.
    \item If $d_Q=0$, we have $\sum_{j_Q=1}^N \abs*{x_{j_Q}}^{d_Q}=N$.
\end{enumerate}
Therefore, we have 
\begin{align}
    \alpha(\cP,I) = \sum_{Q\in\cQ_{\cP,I}} \left(\max\left\{\frac{d_Q}{2},1\right\} - \frac{d_Q}{2}\right) = \sum_{Q\in\cQ_{\cP,I}} \max\left\{0,1- \frac{d_Q}{2}\right\}. \label{eq:alpha(P,I)}
\end{align}
On the other hand, defining the meet of $\cQ_{\cP,I}$ and $\cP\times\{1,2\}$ by
\begin{align*}
    \cQ_{\cP,I}\wedge(\cP\times\{1,2\}) = 
    \{Q\cap (P\times\{1,2\})
    :
    Q\in\cQ_{\cP,I},
    P\in\cP,
    Q\cap (P\times\{1,2\})\neq \varnothing
    \},
\end{align*}
we see that the summation 
\begin{align*}
I(Q\cap (P\times\{1,2\}))
=
\sum_{r\in I(Q\cap (P\times\{1,2\}))}
\abs*{P} 
-
\abs*{I^{-1}(r)\cap (P\times\{1,2\})}
\end{align*}
depends only on the choice of $Q\cap (P\times\{1,2\})\in \cQ_{\cP,I}\wedge(\cP\times\{1,2\})$.
Moreover, defining $\cQ_P = \{Q\in\cQ_{\cP,I}: Q\cap (P\times\{1,2\})\neq \varnothing\}$,
\begin{align*}
    \sum_{Q\in \cQ_P} I(Q\cap (P\times\{1,2\}))
    &=
    \sum_{r\in I(P\times\{1,2\})}
    \abs*{P} 
-
\abs*{I^{-1}(r)\cap (P\times\{1,2\})}
    \nonumber \\
    &=
    \abs*{P}\abs*{I(P\times\{1,2\})}
    -
    \abs*{P\times\{1,2\}}.
\end{align*}
Thus, the sum of $d_Q$ over $Q\in\cQ_{\cP,I}$ equals 
\begin{align}
\sum_{Q\in\cQ_{\cP,I}} d_Q
&= 
\sum_{Q\in\cQ_{\cP,I}}
\sum_{P\in\cP_Q} 
I(Q\cap (P\times\{1,2\}))
\nonumber \\
&=
\sum_{Q\cap P\times\{1,2\}\in \cQ_{\cP,I}\wedge(\cP\times\{1,2\})}
I(Q\cap (P\times\{1,2\}))
\nonumber \\
&=
\sum_{P\in\cP}
\sum_{Q\in\cQ_P}
I(Q\cap (P\times\{1,2\}))
=
-2k + \sum_{P\in\cP}\abs*{P}\abs*{I(P\times\{1,2\})}. \nonumber
\end{align}
Hence,
\begin{align*}
    \prod_{Q\in \cQ_{\cP,I}}\sum_{j_Q=1}^N \abs*{x_{j_Q}}^{d_{Q}}
    \leq N^{\frac{1}{2}\sum_{Q\in\cQ_{\cP,I}} d_Q + \alpha(\cP,I)}
    = N^{-k + \frac{1}{2}\sum_{P\in\cP}\abs*{P}\abs*{I(P\times\{1,2\})} + \alpha(\cP,I)}.
\end{align*}

\subsubsection{Estimate of \texorpdfstring{$\alpha(\cP,I)$}{a(P,I)}}

We now want to prove that for any partition $\cP$ on $\llbracket 1,k\rrbracket$ and $I\in\mathcal{I}_k$,
\begin{align}
    \alpha(\cP,I) \leq \frac{k}{2} \label{eq:a(P,I) bnd}.
\end{align}
 We start with the following lemma that gives a first bound on $\alpha(\cP,I)$ via a partition of $\cQ_{\cP,I}$. Note that the lemma holds for any partition $\cP$ on $\llbracket 1,k\rrbracket$ and $I\in\mathcal{I}_k$. 
\begin{lemma}
    \label{lem:simple claim.1}
    For any partition $\cP$ on $\llbracket 1,k\rrbracket$ and $I\in\mathcal{I}_k$, every block in $\cQ_{\cP,I}$ has even cardinality. In particular, this implies that $\cQ_{\cP,I}$ admits a partition $\cQ_{\cP,I} = \cQ_{2,0}\cup \cQ_{2,1}\cup \cQ_{4,0}\cup \cQ_{4,1}$ where 
\begin{align*}
\cQ_{2,0} &= \{Q\in\cQ_{\cP,I}:\abs{Q}=2,\, d_Q=0\}, \quad
\cQ_{2,1} = \{Q\in\cQ_{\cP,I}:\abs{Q}=2,\, d_Q\geq 1\}, \nonumber \\
\cQ_{4,0} &= \{Q\in\cQ_{\cP,I}:\abs{Q}\geq 4,\, d_Q=0 \}, \quad 
\cQ_{4,1} = \{Q\in\cQ_{\cP,I}:\abs{Q}\geq 4,\, d_Q\geq 1 \}.
\end{align*}
Moreover, the following inequalities hold:
\begin{align}
\abs{\cQ_{2,0}} + \abs{\cQ_{2,1}} + 2\abs{\cQ_{4,0}} + 2 \abs{\cQ_{4,1}}
\leq k \label{eq:2k bound}
\end{align}
and
\begin{align}
    \alpha(\cP,I)
    \leq 
    \abs{\cQ_{2,0}} + \frac{1}{2} \abs{\cQ_{2,1}} + \abs{\cQ_{4,0}} + \frac{1}{2} \abs{\cQ_{4,1}}. \label{eq:alpha(P,I).1}
\end{align}
\end{lemma}
\begin{proof}[Proof of Lemma~\ref{lem:simple claim.1}]
    The first statement in the lemma follows from the fact that the relation $\mathfrak{R}_1$ induces a bijection between the even and odd elements in each block of $\cQ_{\cP,I}$, and the second statement simply follows. 
    Inequality \eqref{eq:2k bound} follows from the partition $\cQ_{\cP,I} = \cQ_{2,0}\cup \cQ_{2,1}\cup \cQ_{4,0}\cup \cQ_{4,1}$ and the following estimate:
    \begin{align*}
    2\abs{\cQ_{2,0}} + 2\abs{\cQ_{2,1}} + 4\abs{\cQ_{4,0}} + 4 \abs{\cQ_{4,1}}
&\leq \sum_{Q\in \cQ_{2,0}} 2
    + \sum_{Q\in \cQ_{2,1}} 2 
    + \sum_{Q\in \cQ_{4,0}} \abs{Q}
    + \sum_{Q\in \cQ_{4,1}} \abs{Q} \nonumber \\
    &= \sum_{Q\in\cQ_{\cP,I}}\abs{Q} \nonumber \\
    &= \abs{\llbracket 1,k\rrbracket\times\{1,2\}} = 2k.
    \end{align*}
    Finally, inequality \eqref{eq:alpha(P,I).1} follows from the partition $\cQ_{\cP,I} = \cQ_{2,0}\cup \cQ_{2,1}\cup \cQ_{4,0}\cup \cQ_{4,1}$ with \eqref{eq:alpha(P,I)}, which yields
\begin{align}
    \alpha(\cP,I)
    &= \sum_{Q\in \cQ_{2,0}} \max\left\{0,1- \frac{d_Q}{2}\right\}
    + \sum_{Q\in \cQ_{2,1}} \max\left\{0,1- \frac{d_Q}{2}\right\} \nonumber \\
    &+ \sum_{Q\in \cQ_{4,0}} \max\left\{0,1- \frac{d_Q}{2}\right\}
    + \sum_{Q\in \cQ_{4,1}} \max\left\{0,1- \frac{d_Q}{2}\right\} \nonumber \\
    &\leq \abs{\cQ_{2,0}} + \frac{1}{2} \abs{\cQ_{2,1}} + \abs{\cQ_{4,0}} + \frac{1}{2} \abs{\cQ_{4,1}}, \nonumber 
\end{align}
and the proof is completed.
\end{proof}

If $\cQ_{2,0}$ is empty, then we can already derive from \eqref{eq:2k bound} and \eqref{eq:alpha(P,I).1}~in~Lemma~\ref{lem:simple claim.1} that 
\begin{align*}
    \alpha(\cP,I)\leq \frac{1}{2} \abs{\cQ_{2,1}} + \abs{\cQ_{4,0}} + \frac{1}{2} \abs{\cQ_{4,1}}
    \leq \frac{1}{2} (\abs{\cQ_{2,1}} + 2\abs{\cQ_{4,0}} + 2 \abs{\cQ_{4,1}})
    \leq \frac{k}{2}.
\end{align*}

Suppose from now on that $Q_{2,0}$ is non-empty. The following lemma characterizes the blocks in $\cQ_{2,0}$ and $\cQ_{2,0} \cup Q_{2,1}$.
\begin{lemma}
\label{lem:Q=2}
If $Q\in\cQ_{2,0}\cup \cQ_{2,1}$, then there exists $\ell\in\llbracket 1,k\rrbracket$ such that
\begin{align}
Q
=
\{(\ell+1,1),(\ell,2)\},
\label{eq:Q}
\end{align}
with the convention that $(k+1,1)=(1,1)$.
Moreover, if $Q\in\cQ_{2,0}$, then $I(\ell+1,1)=I(\ell,2)$ and $Q \subseteq P_Q\times\{1,2\}$ where
\begin{align}
P_Q=
\{\ell,\ell+1\}, \label{eq:P_Q}
\end{align}
with the convention that $k+1=1$.
\end{lemma}
\begin{proof}[Proof of Lemma~\ref{lem:Q=2}]
We start with the proof of \eqref{eq:Q}. Suppose that $(m,i)\in Q\in \cQ_{2,0}\cup\cQ_{2,1}$. 

If $i=1$, then since $((m,1),(m+1,2))\in \mathfrak{R}_1$, $(m+1,2)\in Q$, with the convention that $(k+1,2)=(1,2)$, and $Q=\{(m,1),(m+1,2)\}$.

If $i=2$, then since $((m-1,1),(m,2))\in \mathfrak{R}_1$, $(m-1,1)\in Q$, with the convention that $(0,1)=(k,1)$, and $Q=\{(m-1,1),(m,2)\}$.

Now, suppose that $Q = \{(\ell+1,1),(\ell,2)\}\in \cQ_{2,0}$ for some $\ell\in\llbracket 1,k\rrbracket$, which is guaranteed by \eqref{eq:Q}. By the definition of $d_Q$ in the statement of Lemma~\ref{lem:T upp bnd}, we have that 
\begin{align*}
0
=
d_Q
= 
\sum_{P\in\cP_Q}\sum_{r\in I(Q\cap (P\times\{1,2\}))}
(\abs*{P}-\abs*{I^{-1}(r)\cap (P\times\{1,2\})}).
\end{align*}
Thus, Lemma~\ref{lem:card} yields that for all $P\in\cP_Q$ and $r\in I(Q\cap (P\times\{1,2\}))$, 
\begin{align*}
\abs*{P}=\abs*{I^{-1}(r)\cap (P\times\{1,2\})},
\end{align*}
and since $I^{-1}(r)\cap (P\times\{1,2\})\subseteq Q$, we derive that
\begin{align}
    2=\abs{Q}\geq \abs*{I^{-1}(r)\cap P\times\{1,2\}}=\abs*{P}\geq 2. \label{eq:unique P_Q}
\end{align}
But \eqref{eq:unique P_Q} yields that there exists a unique $P_Q\in\cP_Q$ and $r\in I(Q\cap (P\times\{1,2\}))$ such that 
\begin{align*}
    Q=I^{-1}(r)\cap (P_Q\times\{1,2\})
    \subseteq P_Q\times\{1,2\}.
\end{align*} 
The inclusion above and the fact that $\abs{Q}=2$ yield that either $((\ell+1,1),(\ell,2))$ or $((\ell,2),(\ell,1))$ belongs to $\mathfrak{R}_2$, which both indicates that $I(\ell+1,1)=I(\ell,2)=r$. 
Moreover, since we know from \eqref{eq:unique P_Q} that $\abs*{P_Q\times\{1,2\}}=4$, the inclusion above yields
\begin{align*}
 P_Q\times\{1,2\}=\{(\ell,1),(\ell,2),(\ell+1,1),(\ell+1,2)\},
\end{align*}
which yields \eqref{eq:P_Q}.
\end{proof}

\subsection{Induction on blocks satisfying Condition~\protect\hyperlink{cond:S}{(S)}}

Recall that we assume that $\cQ_{2,0}$ is non-empty. In this section, we want to proceed with the proof of \eqref{eq:a(P,I) bnd} by induction on the number of $Q\in\cQ_{2,0}$ satisfying the following condition.
\begin{enumerate}
    \item[\hypertarget{cond:S}{(S)}]
    Recall that Lemma~\ref{lem:Q=2} implies that for $Q\in\cQ_{2,0}$, there exists $\ell\in\llbracket 1,k\rrbracket$ such that $Q
    =\{(\ell+1,1),(\ell,2)\}$. We say that $Q$ satisfies Condition $(\mathrm{S})$ if 
    one of the following conditions hold:
    \begin{enumerate}
        \item[\hypertarget{cond:S1}{(S1)}] $I(\ell,1)=I(\ell+1,2)$.
        \item[\hypertarget{cond:S2}{(S2)}]
        for all $Q'\in\cQ_{\cP,I}$ such that $Q'\cap \{(\ell,1), (\ell+1,2)\}\neq \varnothing$, $d_{Q'}\geq 2$.
    \end{enumerate}
\end{enumerate}

\subsubsection*{Base case: there are no blocks satisfying Condition~\protect\hyperlink{cond:S}{(S)}}
In this case, for every block $Q
    =\{(\ell+1,1),(\ell,2)\}\in\cQ_{2,0}$, both Condition~\protect\hyperlink{cond:S1}{(S1)} and Condition~\protect\hyperlink{cond:S2}{(S2)} are false. That is,
    \begin{enumerate}
    \item  $I(\ell,1)\neq I(\ell+1,2)$ and
    \item there exists $Q'\in\cQ_{\cP,I}$ with $Q'\cap \{(\ell,1), (\ell+1,2)\}\neq \varnothing$, $d_{Q'}= 1$
    \end{enumerate}
Note that Point~1 implies that $d_{Q'}\neq 0$ in Point~2. 

We claim that in this case, the following lemma holds. 
\begin{lemma}
    \label{lem:2<=4}
    In the base case, $\abs{\cQ_{2,0}}\leq \abs{\cQ_{4,1}}$.
\end{lemma}
\begin{proof}
    We construct an injection $\iota:\cQ_{2,0}\hookrightarrow\cQ_{4,1}$ defined as $\iota(Q)=Q'$ where $Q=\{(\ell+1,1),(\ell,2)\}$ and $Q'\in\cQ_{\cP,I}$ is the block provided by Point 2 of the base case. If there are two distinct choices of $Q'$, we choose the one containing $(\ell,1)$. 

We claim that $\iota(Q)\in \cQ_{4,1}$. Suppose not. 
Then, $\iota(Q)=\{z_1,z_2\}\in \cQ_{2,1}$. 
Denoting by $P_Q = \{\ell,\ell+1\}$, 
let $\{z_1\}= (P_Q\times \{1,2\})\cap \iota(Q)$ and let $P'\in\cP_{\iota(Q)}=\{P'\in\cP:P'\cap \iota(Q)\neq \varnothing\}$ such that $\{z_2\}= (P'\times\{1,2\})\cap \iota(Q)$. 

If $P_Q = P'$, then $\iota(Q) = \{(\ell,1),(\ell+1,2)\}$ and 
\begin{align*}
d_{\iota(Q)}
=\sum_{r\in\{I(\ell,1),I(\ell+1,2)\}} \abs{P_Q} - 
\abs{I^{-1}(r)\cap (P_Q\times\{1,2\})} = 2,
\end{align*}
which violates the assumption that $d_{\iota(Q)}=1$.

If $P_Q\neq P'$, then $I((P_Q\times\{1,2\})\cap\iota(Q)) = \{I(z_1)\}$ and $I((P'\times\{1,2\})\cap\iota(Q)) = \{I(z_2)\}$, Thus, 
\begin{align*}
    d_{\iota(Q)}
    = \sum_{P''\in \{P_Q,P'\}} \sum_{r\in I((P''\times\{1,2\}\cap \iota(Q)))} \abs*{P''} - \abs*{I^{-1}(r)\cap (P''\times\{1,2\}))} 
    \geq 2
\end{align*}
which again leads to a contradiction.

It remains to show that $\iota$ is an injection. Suppose that $Q,Q'\in \cQ_{2,0}$ are distinct and $\iota(Q)=\iota(Q')\in \cQ_{4,1}$. We denote by $P_Q,P_{Q'}\in\cP$ the blocks provided by Lemma~\ref{lem:Q=2} corresponding to $Q$ and $Q'$, respectively. Then, we have $\iota(Q)\cap P_Q = \{z\}$ and $\iota(Q)\cap P_{Q'}=\iota(Q')\cap P_{Q'} = \{z'\}$ and 
\begin{align*}
\abs*{P_Q} - \abs*{I^{-1}(I(z))\cap (P_Q\times \{1,2\})}=1,
\quad
\abs*{P_{Q'}} - \abs*{I^{-1}(I(z'))\cap (P_{Q'}\times \{1,2\})}=1,
\end{align*}
and the equations above yield
\begin{align*}
    1=d_{\iota(Q)}
    \geq
    \sum_{P\in\{P_Q,P_{Q'}\}} 
    \sum_{r\in I((P\times\{1,2\})\cap Q)} 
    \abs*{P} - \abs*{I^{-1}(r)\cap (P\times\{1,2\})}
    =2,
\end{align*}
which is a contradiction.
\end{proof}
By Lemma~\ref{lem:2<=4}, \eqref{eq:2k bound} and \eqref{eq:alpha(P,I).1}~in~Lemma~\ref{lem:simple claim.1} yield 
\begin{align*}
    \alpha(\cP,I)
    &\leq \abs{\cQ_{2,0}} + \frac{1}{2} \abs{\cQ_{2,1}} + \abs{\cQ_{4,0}} + \frac{1}{2} \abs{\cQ_{4,1}} \nonumber \\
    &= \frac{1}{2} (\underbrace{\abs{\cQ_{2,0}} + \abs{\cQ_{2,1}} + 2\abs{\cQ_{4,0}} + 2 \abs{\cQ_{4,1}}}_{\leq k}) + \frac{1}{2} (\underbrace{\abs{\cQ_{2,0}}-\abs{\cQ_{4,1}}}_{\leq 0}) 
    \leq \frac{k}{2}.
\end{align*}

\subsubsection*{Induction step} 

Suppose that there exists $Q_0=\{(\ell+1,1),(\ell,2)\}\in\cQ_{2,0}$ satisfying Condition~\hyperlink{cond:S}{(S)}.  
Recall that Lemma~\ref{lem:Q=2} implies that there exists $P_{Q_0} = \{\ell,\ell+1\}\in\cP$ such that $Q_0\subseteq P_{Q_0}$.

We distinguish the following two cases.

\paragraph{Case 1: Condition~\protect\hyperlink{cond:S1}{(S1)} holds.} 
 We define the bijection $f:\llbracket 1,k\rrbracket \backslash P_{Q_0}\rightarrow \llbracket 1,k-2\rrbracket$ by
\begin{align*}
    f(m)
    =
    \begin{cases}
        m, & m\in \llbracket 1,\ell-1\rrbracket, \\
        m-2, & m\in \llbracket\ell+2,k\rrbracket.
    \end{cases}
\end{align*}
Define $\cP' = \{f(P):P\in \cP\backslash \{P_{Q_0}\}\}$. Since $f$ is a bijection and $P_{Q_0}$ is a block of size $2$, we see that $\cP'$ is a partition on $\llbracket 1,k-2\rrbracket$ such that $\min_{P\in\cP'}|P|\geq 2$ and  $\max_{P\in\cP'}|P|>2$. 

Let $I'$ be a mapping defined on $\llbracket 1,k-2\rrbracket\times\{1,2\}$ defined by
\begin{align*}
    I'(m,i) = I(f^{-1}(m),i), \quad (m,i)\in \llbracket 1,k-2\rrbracket\times\{1,2\}.
\end{align*}
Then, one can check that $I'\in\mathcal{I}_{k-2}$.

Define $\cQ_{\cP',I'}$ to be the partition on $\llbracket 1,k-2\rrbracket\times \{1,2\}$ defined by the relation $\mathfrak{R}_1'\cup \mathfrak{R}_2'$ where
    \begin{align*}
        &\mathfrak{R}_1'
        \coloneqq \bigcup_{m=1}^{k-2}\{((m+1,1),(m,2))\},
        \nonumber \\
        &\mathfrak{R}_2'
        \coloneqq
        \bigcup_{P\in\cP'} 
        \Big\{
        ((m_1,i_1),(m_2,i_2)) \in (P\times\{1,2\})^2
        :
        I'(m_1,i_1) = I'(m_2,i_2) 
        \Big\},
    \end{align*}
    with the convention that $(k-1,1) = (1,1)$. 

    Note that $\cQ_{\cP',I'}$ is a partition with one less block satisfying Condition~\hyperlink{Cond:S}{(S)}. Thus, by the induction hypothesis, we have $\alpha(\cP',I')\leq (k-2)/2$. We claim that 
  \begin{align}
      \alpha(\cP,I)= \alpha(\cP',I')+1, \label{eq:increment of alpha}
  \end{align}
  from which we can conclude that $\alpha(\cP,I)\leq k/2$. To prove \eqref{eq:increment of alpha}, we start with a lemma that characterizes $Q_{\cP',I'}$ in terms of $f$, $\cQ_{\cP,I}$ and $\cP$.
\begin{lemma}
\label{lem:retraction} 
Assume that Condition~\protect\hyperlink{cond:S1}{(S1)} holds, denote by $Q_1$ the unique block in $\cQ_{\cP,I}$ containing $(\ell,1),(\ell+1,2)$. Then 
\begin{align*}
\cQ_{\cP',I'} = 
\{
f(Q):Q\in\cQ_{\cP,I}\backslash\{Q_0,Q_1\}
\}
\cup
\{f(Q_1\backslash  (P_{Q_0}\times\{1,2\}))\},
\end{align*}
where $f(Q) \coloneqq \{(f(m),i):m\in Q\}$ for any subset $Q\subseteq (\llbracket 1,k\rrbracket\backslash Q_0) \times\{1,2\}$.
\end{lemma}
\begin{proof}
Defining $w_1 = (\ell+1,1) $ and $w_2 = (\ell,2)$, recall that $Q_0 = \{w_0,w_1\}$.

First of all, note that the relation $\mathfrak{Q}_1\cup\mathfrak{Q}_2$ where
\begin{align*}
    \mathfrak{Q}_1
    =
    \mathfrak{R}_1\backslash\{(w_1,w_2)\}, 
    \qquad
    \mathfrak{Q}_2
    =
    \mathfrak{R}_2\backslash\{(w_1,w_2),(w_2,w_1)\}
\end{align*}
generates the partition $\cQ_{\cP,I}\backslash\{Q_0\}$ on $(\llbracket 1,k\rrbracket \times \{1,2\})\backslash Q_0$. 

Moreover, defining 
\begin{align*}
    v_1 = (\ell,1),
    \qquad
    v_2 = (\ell-1,2),
    \qquad
    u_1 = (\ell+2,1),
    \qquad
    u_2 = (\ell+1,2),
\end{align*}
let $\tilde{Q}$ be partition on
\begin{align*}
    (\llbracket 1,k\rrbracket \times \{1,2\})\backslash (P_{Q_0}\times\{1,2\}) = (\llbracket 1,k\rrbracket \times \{1,2\})
\backslash\{v_1,v_2,u_1,u_2\}
\end{align*}
generates by
the relation $\mathfrak{Q}_3\cup\mathfrak{Q}_4$
where
\begin{align*}
    \mathfrak{Q}_3
    =
    \mathfrak{Q}_1
    \backslash
    \{(v_1,v_2),(u_1,u_2)\},
    \qquad
    \mathfrak{Q}_4 
    =
    \mathfrak{Q}_2
    \backslash
    \{(v_1,u_2),(u_2,v_1)\}.
\end{align*} 
Denoting by $\sim$ the equivalence relation generated by $\mathfrak{Q}_3\cup\mathfrak{Q}_4$, 
let
\begin{align*}
    Q_{v_2} = \{v\in Q_1 : v\sim v_2\},
    \qquad 
    Q_{u_1} = \{v\in Q_1 : v\sim u_1\}.
\end{align*}
\begin{itemize}
    \item If $v_2\sim u_1$, then $Q_{v_2} = Q_{u_1} = Q_1\backslash (P_{Q_0}\times \{1,2\})$ and 
    \begin{align*}
        \cQ_1 = (\cQ_{\cP,I} \backslash \{Q_0,Q_1\})\cup \{ Q_1\backslash (P_{Q_0}\times \{1,2\})\}.
    \end{align*}
    Moreover, since $(v_2,u_1)\in (\mathfrak{Q}_3 \cup (v_2,u_1))\cup\mathfrak{Q}_4$, $\tilde{Q}$ is also the partition generated by $(\mathfrak{Q}_3 \cup (v_2,u_1))\cup\mathfrak{Q}_4$
    \item If $v_2\not\sim u_1$, then $Q_{v_2} \neq Q_{u_1}$ and 
    \begin{align*}
        \cQ_1 = (\cQ_{\cP,I} \backslash \{Q_0,Q_1\})\cup \{ Q_{v_2}, Q_{u_1}\}.
    \end{align*}
    Since $(v_2,u_1)\in (\mathfrak{Q}_3 \cup (v_2,u_1))\cup\mathfrak{Q}_4$, 
    \begin{align*}
        \cQ_2
        =
        (\cQ_{\cP,I} \backslash \{Q_0,Q_1\})\cup \{ Q_{v_2}\cup Q_{u_1}\}
        =
        (\cQ_{\cP,I} \backslash \{Q_0,Q_1\})\cup \{ Q_1\backslash (P_{Q_0}\times \{1,2\})\}
    \end{align*}
    is the partition generated by $(\mathfrak{Q}_3 \cup (v_2,u_1))\cup\mathfrak{Q}_4$.
\end{itemize}
In conclusion, the relation $(\mathfrak{Q}_3 \cup (v_2,u_1))\cup\mathfrak{Q}_4$ generates the partition 
    \begin{align*}
        \tilde{Q}
        =
        (\cQ_{\cP,I} \backslash \{Q_0,Q_1\})\cup \{ Q_1\backslash (P_{Q_0}\times \{1,2\})\}
    \end{align*}
    on $(\llbracket 1,k\rrbracket \times \{1,2\})\backslash (P_{Q_0}\times\{1,2\})$.
    Since $(m,i)\mapsto (f(m),i)$ is a bijection from $(\llbracket 1,k\rrbracket \times \{1,2\})\backslash (P_{Q_0}\times\{1,2\})$ to $\llbracket 1,k-2\rrbracket \times \{1,2\}$,
    \begin{align*}
        f(\tilde{\cQ})
        =
        \{
        f(Q):Q\in\cQ_{\cP,I}\backslash\{Q_0,Q_1\}
        \}
        \cup
        \{f(Q_1\backslash  (P_{Q_0}\times\{1,2\}))\}
    \end{align*}
    is a partition on $\llbracket 1,k-2\rrbracket \times \{1,2\}$. Moreover, one can check that
    \begin{align*}
        &f(\mathfrak{Q}_3\cup \{v_2,u_1\}) \nonumber \\
        &\coloneqq  
        \big\{
        ((f(\ell+2),1),(f(\ell-1),2))
        \big\} 
        \cup
        \bigcup_{m\in\llbracket 1,k\rrbracket\rrbracket \backslash \{\ell-1,\ell,\ell+1\}}
        \big\{
        ((f(m+1),1),(f(m),2)
        \big\} 
        \nonumber \\
        &=
        \mathfrak{R}_1'
    \end{align*}
    and 
    \begin{align*}
        f(\mathfrak{Q}_4)
        &\coloneqq  
        \bigcup_{P\in\cP\backslash\{P_{Q_0}\}}
        \big\{
        ((f(m_1),i_1),(f(m_2),i_2)) \in (f(P)\times\{1,2\})^2
        :
        I(m_1,i_1) = I(m_2,i_2)
        \big\} \\
        &=
        \mathfrak{R}_2',
    \end{align*}
    so we can conclude that $f(\tilde{\cQ}) = \cQ_{\cP',I'}$, which completes the proof.
\end{proof}
By the characterization of $\cQ_{\cP',I'}$ provided by Lemma~\ref{lem:retraction}, to determine $d_{Q'}$ for each $Q'\in\cQ_{\cP',I'}$, it suffices to consider the following two cases:
\begin{itemize}
\item Recall the definition of $I'$ and the fact that $f$ is a bijection. For all $P\in\cP_Q\backslash\{P_{Q_0}\}$, where we recall that $\cP_Q = \{P\in\cP:(P\times\{1,2\})\cap Q\neq \varnothing\}$,
    we have $\abs{f(P)}=\abs{P}$, 
    \begin{align*}
        I'(f(Q)\cap (f(P)\times \{1,2\}))
        = I(Q\cap (P\times \{1,2\}))
    \end{align*}
    and 
    \begin{align*}
        \abs*{(I')^{-1}(r)\cap (f(P)\times\{1,2\})})
        =
        \abs*{I^{-1}(r)\cap (P\times\{1,2\})})
    \end{align*}
    for $r\in I(Q\cap (P\times \{1,2\}))$. Therefore, for all $Q\in \cQ_{\cP,I}\backslash\{Q_0,Q_1\}$, 
    \begin{align}
        d_{f(Q)}
        &= 
        \sum_{f(P)\in\cP'_{f(Q)}}\sum_{r\in I'(f(Q)\cap (f(P)\times\{1,2\}))}
        (\abs*{f(P)}-\abs*{(I')^{-1}(r)\cap (f(P)\times\{1,2\})}) 
        \nonumber \\
        &=
        \sum_{P\in\cP_{Q}\backslash \{P_{Q_0}\}}\sum_{r\in I(Q\cap (P\times\{1,2\}))}
        (\abs*{P}-\abs*{I^{-1}(r)\cap (P\times\{1,2\})})
        =
        d_Q. \nonumber 
    \end{align}
\item Now, let $Q=Q_1\backslash (P_{Q_0}\times \{1,2\})$.
    Note that for all $P\in \cP$ such that $P\cap Q_1\neq \varnothing$, if $P\neq P_{Q_0}$, then $P$ and $P_{Q_0}$ are disjoint, and this yields
    \begin{align}
        (P\times\{1,2\})\cap (Q_1\backslash(P_{Q_0}\times\{1,2\})
        =
        (P\times\{1,2\})\cap Q_1 \label{eq:PPQ0}
    \end{align}
    Then, \eqref{eq:PPQ0} implies that 
    \begin{align}
        \cP_Q\backslash \{P_{Q_0}\}
        &=
        \{P\in\cP : (P\times\{1,2\})\cap (Q_1\backslash(P_{Q_0}\times\{1,2\}) 
        \neq
        \varnothing\}
        \backslash \{P_{Q_0}\}
        \nonumber \\
        &=
        \{P\in\cP : (P\times\{1,2\})\cap Q_1
        \neq
        \varnothing\}
        \backslash \{P_{Q_0}\}
        \nonumber \\
        &=
        \cP_{Q_1}\backslash \{P_{Q_0}\} \label{eq:PPQ0.1}
    \end{align}
    and 
    \begin{align}
        I((P\times\{1,2\})\cap (Q_1\backslash(P_{Q_0}\times\{1,2\}))
        =
        I((P\times\{1,2\})\cap Q_1). \label{eq:PPQ0.2}
    \end{align}
    From \eqref{eq:PPQ0.1} and \eqref{eq:PPQ0.2}, we obtain that 
    \begin{align}
        d_{f(Q)}
        &= 
        \sum_{f(P)\in\cP'_{f(Q)}}\sum_{r\in I'(f(Q)\cap (f(P)\times\{1,2\}))}
        (\abs*{f(P)}-\abs*{(I')^{-1}(r)\cap (f(P)\times\{1,2\})}) 
        \nonumber \\
        &=
        \sum_{P\in\cP_{Q}\backslash\{P_{Q_0}\}}\sum_{r\in I(Q\cap (P\times\{1,2\}))}
        (\abs*{P}-\abs*{I^{-1}(r)\cap (P\times\{1,2\})})
        \nonumber \\
        &=
        \sum_{P\in\cP_{Q_1}\backslash\{P_{Q_0}\}}
        \sum_{r\in I(Q_1\cap (P\times\{1,2\}))}
        (\abs*{P}-\abs*{I^{-1}(r)\cap (P\times\{1,2\})}).
        \label{eq:df(Q).1}
    \end{align}
    On the other hand,
    \begin{align}
        d_{Q_1}
        = 
        \sum_{P\in\cP_{Q_1}}\sum_{r\in I(Q_1\cap (P\times\{1,2\}))}
        (\abs*{P}-\abs*{I^{-1}(r)\cap (P\times\{1,2\})}) 
        \label{eq:dQ_1},
    \end{align}
    Thus, by \eqref{eq:df(Q).1} and \eqref{eq:dQ_1}, subtracting $d_{f(Q)}$ from $d_{Q_1}$ yields
    \begin{align*}
        d_{Q_1} - d_{f(Q)}
        =
        \sum_{r\in I(Q_1\cap (P_{Q_0}\times\{1,2\}))}
        (\abs*{P_{Q_0}}-\abs*{I^{-1}(r)\cap (P_{Q_0}\times\{1,2\})}).
    \end{align*}
    Recall that $Q_1\cap (P_{Q_0}\times\{1,2\} = \{(\ell,1),(\ell+1,2)\}$. 
    Since we assume Condition~\hyperlink{Cond:S1}{(S1)} holds, we have $I(\ell,1) = I(\ell+1,2)$ which yields
    \begin{align*}
        d_{Q_1} - d_{f(Q)}
        =
        (\abs*{P_{Q_0}}-\abs*{I^{-1}(I(\ell,1))\cap (P_{Q_0}\times\{1,2\})})
        =
        2 - \abs{\{(\ell,1),(\ell+1,2)\}}
        = 0.
    \end{align*}
\end{itemize}
By the case study above, we conclude that 
\begin{align*}
    \alpha(\cP,I)-\alpha(\cP',I') 
    =
    \max\Big\{0,1-\frac{d_{Q_0}}{2}\Big\} = 1,
\end{align*}
which proves \eqref{eq:increment of alpha}.

\paragraph{Case 2: Condition~\protect\hyperlink{Cond:S1}{(S1)} does not hold, but Condition~\protect\hyperlink{Cond:S2}{(S2)} holds.} Adopting the abbreviations  $v_1 = (\ell,1)$ and $v_2 = (\ell+1,2)$, the assumption of Case 2 can be stated explicitly as the following two points:
\begin{enumerate}
    \item $I(v_1)\neq I(v_2)$.
    \item For all $Q'\in\cQ_{\cP,I}$ such that $Q'\cap \{v_1,v_2\}\neq \varnothing$, $d_{Q'}\geq 2$.
\end{enumerate} 
To treat Case~2, we want to keep $\cP$ but modify $I$ to be $I'$ such that Condition~\protect\hyperlink{Cond:S1}{(S1)} holds for $\cQ_{\cP,I'}$, and 
\begin{align}
    \alpha(\cP,I) = \alpha(\cP,I').
    \label{eq:alpha equal}
\end{align}
Then, from \eqref{eq:alpha equal} and Case 1, we conclude that $\alpha(\cP,I)\leq k/2$.
We define 
\begin{align*}
    I'(m,i) =
    \begin{cases}
        I(m,i), & (m,i) \neq v_1,v_2, \\
        I(v_1), & (m,i) = v_1 \  \text{or} \ v_2.
    \end{cases}
\end{align*}
One can easily check that $I'\in\mathcal{I}_{k}$. Thus, it remains to prove \eqref{eq:alpha equal}.

Let $Q_1,Q_2\in\cQ_{\cP,I}$ contain $v_1$ and $v_2$, respectively. 

If $Q_1=Q_2$, then $\cQ_{\cP,I}=\cQ_{\cP,I'}$, and so $\alpha(\cP,I)=\alpha(\cP,I)$.

If $Q_1\cap Q_2 = \varnothing$, then $\cQ_{\cP,I'}=(\cQ_{\cP,I}\backslash\{Q_1,Q_2\})\cup \{Q_1\cup Q_2\}$. For all $ Q\in\cQ_{\cP,I'}$, denoting by 
\begin{align*}
    d_Q'
    =\sum_{P\in\cP_Q}\sum_{r\in I'(Q\cap (P\times\{1,2\}))} \abs*{P} - \abs*{(I')^{-1}(r)\cap P}, 
\end{align*}
we have that $d_Q'=d_Q$ for all $Q\neq Q_1,Q_2$. 

Define $U(P,Q,I) = \sum_{r\in I(Q\cap (P\times\{1,2\}))}
    \abs*{P}
    -
    \abs*{I^{-1}(r)\cap P}$
for $P\in \cP$ and $Q\in\cQ_{\cP,I}$. 
Point~1 of the case assumption yields that for $i=1,2$,
\begin{align*}
    I(Q_i\cap (P_{Q_0}\times\{1,2\}))=\{I(v_i)\},
    \qquad
    I^{-1}(I(v_i)) \cap P_{Q_0} = \{v_i\},
\end{align*}
so
\begin{align*}
    U(P_{Q_0},Q_i,I)
    =
    \abs*{P_{Q_0}} - \abs*{\{v_i\}} = 1.
\end{align*}
On the other hand, recall the Point~2 of the case assumption that $d_{Q_i}\geq 2$ for $i=1,2$. Thus, for $i=1,2$,
\begin{align*}
    \sum_{P\in\cP_{Q_i}\backslash\{P_{Q_0}\}}
    U(P,Q_i,I)
    =
    \sum_{P\in\cP_{Q_i}\backslash\{P_{Q_0}\}}
    \sum_{r\in I(Q\cap (P\times\{1,2\}))}
    \abs*{P}
    -
    \abs*{I^{-1}(r)\cap P}
    \geq 1,
\end{align*}
and this implies that there exist $P_i\in \cP_{Q_i}\backslash\{P_{Q_0}\}$ and $r_i\in I(Q_i\cap (P_i\times\{1,2\}))$, for $i=1,2$, such that 
\begin{align}
    \abs*{P_i}
    -
    \abs*{I^{-1}(r_i)\cap P_i} \geq 1.
    \label{eq:>=1}
\end{align}
Note that $P_1,P_2\in \cP_{Q_1\cup Q_2}$. Moreover, since $P_1,P_2\neq P_{Q_0}$, we have 
    \begin{align*}
    I'((Q_1\cup Q_2)\cap (P_i\times\{1,2\})) 
    &= I((Q_1\cup Q_2)\cap (P_i\times \{1,2\})),
    \end{align*}
    which implies that $r_1,r_2\in I'((Q_1\cup Q_2)\cap (P_i\times\{1,2\}))$. 
Finally, because $P_1,P_2\neq P_{Q_0}$, for $i=1,2$ and $r\in I'((Q_1\cup Q_2)\cap (P_i\times\{1,2\}))$,
    \begin{align*}
    (I')^{-1}(r)\cap P_i 
    &= I^{-1}(r)\cap P_i.
    \end{align*}
We now distinguish the two cases:
\begin{itemize}
    \item Suppose that $P_1\neq P_2$. Then, the paragraph before this case study and \eqref{eq:>=1} yield that
    \begin{align*}
        d_{Q_1\cup Q_2}'
        &=
        \sum_{P\in\cP_{Q_1\cup Q_2}}
        \sum_{r\in I'((Q_1\cup Q_2)\cap (P\times\{1,2\}))}
        \abs*{P}
        -
        \abs*{(I')^{-1}(r)\cap P}
        \\
        &\geq 
        \sum_{i=1,2}
        \abs*{P_i}
        -
        \abs*{I^{-1}(r_i)\cap P_i}
        \geq 2.
    \end{align*}
    \item Suppose that $P_1=P_2$. Since we assume that $Q_1,Q_2$ are disjoint blocks in $\cQ_{\cP,I}$, the definition of $\mathfrak{R}_2$ yields that 
    $I(w_1)\neq I(w_2)$
    for all $w_1\in P_1\cap Q_1$ and $w_2\in P_1\cap Q_2$, and this yields that $r_1\neq r_2$. Thus, the paragraph before this case study and \eqref{eq:>=1} yield
    \begin{align*}
        d_{Q_1\cup Q_2}'
        &=
        \sum_{P\in\cP_{Q_1\cup Q_2}}
        \sum_{r\in I'((Q_1\cup Q_2)\cap (P\times\{1,2\}))}
        \abs*{P}
        -
        \abs*{(I)^{-1}(r)\cap P}
        \\ 
        &\geq 
        \sum_{r\in I'((Q_1\cup Q_2)\cap (P_1\times\{1,2\}))}
        \abs*{P_1}
        -
        \abs*{(I)^{-1}(r)\cap P_1}
        \\
        &\geq 
        \sum_{i=1,2}
        \abs*{P_1}
        -
        \abs*{(I)^{-1}(r_i)\cap P_1}
        \geq 2.
    \end{align*}
\end{itemize}
In conclusion, we see that 
\begin{align}
    \alpha(\cP,I') - \alpha(\cP,I)
    =
    \max\Big\{0,1-\frac{d_{Q_1\cup Q_2}}{2}\Big\}
    -
    \max\Big\{0,1-\frac{d_{Q_1}}{2}\Big\}
    -
    \max\Big\{0,1-\frac{d_{Q_2}}{2}\Big\}
    =0, \nonumber 
\end{align}
which proves \eqref{eq:alpha equal}.

\section{Uniform control on the edge} \label{sec:concentration}

The goal of this section is to prove the following proposition.
\begin{proposition}
    \label{prop:sup edge control}
    For all $c>0$ and $\varepsilon>0$, there exists $\delta>0$ such that
    \begin{align*}
        \sup_{\bm{x}\in\bm{B}^N}
        \PP
        \Big(
        \mu_{N,\bm{x}}((-\infty,-2\xi''(\rho_{\bm{x}})^{1/2}+\varepsilon])
        \leq \delta 
        \Big) \leq e^{-cN},
    \end{align*}
    where $\rho_{\bm{x}} \coloneqq \norm{\bm{x}}^2/N$.
\end{proposition}
We start with the required lemmas, and the proof of Proposition~\ref{prop:sup edge control} is provided in Section~\ref{sec:proof sup edge control}. 
\begin{lemma}
    \label{lem:hp approx}
    For all $c>0$ and $f\in\Lip_1(\mathbb{R})$, there exist $p_0\geq 2$ and $\delta_0>0$ such that for all $p\geq p_0$ and $\delta<\delta_0$,
    \begin{align}
        \sup_{\bm{x}\in\bm{B}^N}
        \PR{
        \abs{
        \int_\R f(y)\dd{\mu_{N,\bm{x}}(y)} -
        \int_\R f(y)\dd{\mu_{N,p,\bm{x}}(y)} 
        }
        \geq \delta
        }
        \leq e^{-cN}. \label{eq:hp approx}
    \end{align}
\end{lemma}
\begin{proof}
    If $\deg\xi<\infty$, we can simply choose $p=\deg\xi$, so assume in the following that $\deg\xi=\infty$.
    
    Given $\bm{x}\in\bm{B}^N$ and $p\geq 2$, by Weyl's inequality, for all $i\in\llbracket 1,N\rrbracket$,
    \begin{align*}
    \abs*{
    \lambda_{N,\bm{x}}^{(i)}
    -
    \lambda_{N,p,\bm{x}}^{(i)}
    }
    \leq
    \sum_{r=p+1}^\infty 
    	\gamma_r
        \norm*{\nabla_E^2 H^{(r)}_{N}(\bm{x})}_{op}
    \end{align*}
    Thus, for all $f\in\Lip_1(\mathbb{R})$,
    \begin{align}
        \abs{
        \int_\R f(y)\dd{\mu_{N,\bm{x}}(y)} -
        \int_\R f(y)\dd{\mu_{N,p,\bm{x}}(y)} 
        }
        \leq 
        \sum_{r=p+1}^\infty 
        \norm*{\nabla_E^2 H^{(r)}_{N}(\bm{x})}_{op}. \label{eq:sum H_p_0}
    \end{align}
    Combining \eqref{eq:sum H_p_0}, Proposition~\ref{prop:Lipschitz} and the union bound, 
    \begin{align}
        &\PR{
        \abs{
        \int_\R f(y)\dd{\mu_{N,\bm{x}}(y)} -
        \int_\R f(y)\dd{\mu_{N,p,\bm{x}}(y)} 
        } 
        \geq 
        \sum_{r=p+1}^\infty \gamma_r r(r-1)2C\sqrt{r}
        } 
        \nonumber \\
        &\leq 
        \sum_{r=p+1}^\infty \exp(-\frac{C^2r}{2K}N), \label{eq:sum H_p_0.1}
    \end{align}
    for $N$ sufficiently large. As the upper bound of \eqref{eq:sum H_p_0.1} does not depend on $\bm{x}$, we have
    \begin{align}
        &\sup_{\bm{x}\in\bm{B}^N}
        \PR{
        \abs{
        \int_\R f(y)\dd{\mu_{N,\bm{x}}(y)} -
        \int_\R f(y)\dd{\mu_{N,p,\bm{x}}(y)} 
        } 
        \geq 
        \sum_{r=p+1}^\infty \gamma_r r(r-1)2C\sqrt{r}
        }
        \nonumber \\
        &\leq 
        \sum_{r=p+1}^\infty \exp(-\frac{C^2r}{2K}N)
        \nonumber \\
        &\leq 
        \exp(-\frac{C^2(p+1)}{2K}N)
        \label{eq:sum H_p_0.2}
    \end{align}
    for $N$ sufficiently large.
    Define 
    \begin{align*}
        \delta_{p} = \sum_{r=p+1}^\infty \gamma_r r(r-1)2C\sqrt{r}
        \quad \text{and} \quad 
        c_p
        =
        \frac{C^2(p+1)}{2K}.
    \end{align*}
    Note that $\delta_{p}$ is non-increasing and $\delta_{p}\rightarrow 0$ as $p\rightarrow\infty$ and $c_p$ is non-decreasing and $c_p\rightarrow\infty$ as $p\rightarrow\infty$. Thus,
    for all $c>0$ and $f\in\Lip_1(\R)$, by taking $p_0\geq 2$ such that $c_{p_0}\geq c$ and $\delta=\delta_{p_0}$, \eqref{eq:hp approx} follows from \eqref{eq:sum H_p_0.2}.
\end{proof}

\begin{lemma}
\label{lem:concentration}
    For all $c>0$, $p\geq 2$ and $f:\Lip_1(\R)$, there exists $\delta>0$ such that
    \begin{align*}
        \sup_{\bm{x}\in\bm{B}^N}
        \PR{
        \abs{
        \int_\R f(y)\dd{\mu_{N,p,\bm{x}}(y)} - 
        \int_\R f(y)\dd{\bar{\mu}_{N,p,\bm{x}}(y)
        }
        }
        \geq \delta
        }
        \leq e^{-cN}.
    \end{align*}
\end{lemma}
\begin{proof}
We define a function $F:\R^{N^2}\times\cdots\times \R^{N^{p_0}}\rightarrow\R$ by
    \begin{align*}
        F(J^{\leq p_0})
        =
        \int_\R f(y)\dd{\mu_{N,\bm{x},J^{\leq p_0}}(y)},
    \end{align*}
    where $\mu_{N,\bm{x},J^{\leq p_0}} = \sum_{i=1}^N \lambda_{N,\bm{x},J^{\leq p_0}}^{(i)}$ is the empirical spectral distribution of $\nabla^2 H_N^{(p_0)}(\bm{x};J^{\leq p_0})\coloneqq \sum_{p=2}^{p_0} H_N^{(p)}(\bm{x};J^{(p)})$.
    
    We have
    \begin{align}
        \abs*{F(J^{\leq p_0}) - F({J'}^{\leq p_0})}
        &\leq
        \sum_{i=1}^N
        \frac{1}{N}
        \abs*{\lambda_{N,\bm{x},J^{\leq p_0}}-\lambda_{N,\bm{x},{J'}^{\leq p_0}}}
        \nonumber \\
        &\leq 
        \norm*{\nabla^2 H^{\leq p_0}_{N}(\bm{x};J^{\leq p_0}) - \nabla^2 H^{\leq p_0}_{N}(\bm{x};{J'}^{\leq p_0})
        }_{op}
        \nonumber \\
        &\leq
        \sum_{p=2}^{p_0} 
        \gamma_p
        \norm*{\nabla^2_E H^{(p)}_{N}(\bm{x};J^{(p)}) - \nabla^2_E H^{(p)}_{N}(\bm{x};{J'}^{(p)})}_{op}.
         \label{eq:Hoffman--Wielandt}
    \end{align}
    where the last line above follows from the triangle inequality and the fact that $\norm{P_{\bm{x}}}_{op}=1$.
    
    For all $p\in\llbracket 2,p_0\rrbracket$ and $\bm{u}\in\R^N$ with $\norm{\bm{u}}=1$, we have
    \begin{align}
        &
        \langle
        (\nabla^2_E H^{(p)}_{N}(\bm{x};J^{(p)}) - \nabla^2_E H^{(p)}_{N}(\bm{x};{J'}^{(p)}))\bm{u},\bm{u}
        \rangle 
        \nonumber \\
        &=
        \frac{1}{N^{(p-1)/2}}
        \sum_{j,k=1}^N
        u_ju_k
        \sum_{i_1,\ldots,i_p=1}^N 
        (J_{i_1,\ldots,i_p}^{(p)} - {J'}_{i_1,\ldots,i_p}^{(p)})
        \partial_{j}\partial_{k}(x_{i_1}\cdots x_{i_p}) 
        \label{eq:Frobenius}
    \end{align}
    By the Cauchy--Schwarz inequality to \eqref{eq:Frobenius}, we have
    \begin{align}
        &
        \left(
        \sum_{i_1,\ldots,i_p=1}^N \left(J_{i_1,\ldots,i_p}^{(p)} - {J'}_{i_1,\ldots,i_{p}}^{(p)}\right) 
        \partial_{j}\partial_{k}(x_{i_1}\cdots x_{i_p})
        \right)^{2} \nonumber \\
        &\leq 
        \left(
        \sum_{i_1,\ldots,i_p=1}^N 
        \left(
        J_{i_1,\ldots,i_p}^{(p)} - {J'}_{i_1,\ldots,i_{p}}^{(p)}
        \right)^2
        \Ind_{j,k\in\{i_1,\ldots,i_p\}}
        \right)
        \sum_{i_1,\ldots,i_p=1}^N 
        \left(
        \partial_{j}\partial_{k}(x_{i_1}\cdots x_{i_p}) 
        \right)^2 \nonumber \\
        &\leq 
        \left(
        \sum_{i_1,\ldots,i_{p}=1}^N 
        \left(J_{i_1,\ldots,i_{p-2},j,k}^{(p)} - {J'}_{i_1,\ldots,i_{p}}^{(p)}
        \right)^2 
        \Ind_{j,k\in\{i_1,\ldots,i_p\}}
        \right)
        \left(
        p(p-1) \sum_{i_1,\ldots,i_{p-2}=1}^N x_{i_1}^2\cdots x_{i_{p-2}}^2
        \right) \nonumber \\
        &= 
        \left(
        \sum_{i_1,\ldots,i_p=1}^N 
        \left(J_{i_1,\ldots,i_p}^{(p)} - {J'}_{i_1,\ldots,i_{p}}^{(p)}
        \right)^2 
        \Ind_{j,k\in\{i_1,\ldots,i_p\}}
        \right)
        p(p-1)\norm{\bm{x}}^{2(p-2)}.
        \label{eq:Cauchy--Schwarz}
    \end{align}
    Plugging \eqref{eq:Cauchy--Schwarz} back to \eqref{eq:Frobenius}, we derive that 
    \begin{align}
        &\langle
        (\nabla^2_E H^{(p)}_{N}(\bm{x};J^{(p)}) - \nabla^2_E H^{(p)}_{N}(\bm{x};{J'}^{(p)}))\bm{u},\bm{u} 
        \rangle
        \nonumber \\
        &\leq 
        \frac{1}{N^{(p-1)/2}}
        (p(p-1))^{1/2}\norm{\bm{x}}^{p-2} 
        \sum_{j,k=1}^N
        u_ju_k
        \left(
        \sum_{i_1,\ldots,i_p=1}^N 
        \left(J_{i_1,\ldots,i_p}^{(p)} - {J'}_{i_1,\ldots,i_{p}}^{(p)}
        \right)^2 
        \Ind_{j,k\in\{i_1,\ldots,i_p\}}
        \right)^{1/2}
        \nonumber \\
        &\leq 
        \frac{1}{N^{(p-1)/2}}
        (p(p-1))\norm{\bm{x}}^{p-2} 
        \norm{\bm{u}}
        \left(
        \sum_{i_1,\ldots,i_p=1}^N 
        \left(J_{i_1,\ldots,i_p}^{(p)} - {J'}_{i_1,\ldots,i_{p}}^{(p)}
        \right)^2 
        \right)^{1/2}. \nonumber \\
        &\leq 
        \frac{1}{\sqrt{N}}
        (p(p-1))
        \left(
        \sum_{i_1,\ldots,i_p=1}^N 
        \left(J_{i_1,\ldots,i_p}^{(p)} - {J'}_{i_1,\ldots,i_{p}}^{(p)}
        \right)^2
        \right)^{1/2}
        , \nonumber 
    \end{align}
    and so
    \begin{align}
        \norm*{\nabla^2_E H^{(p)}_{N}(\bm{x};J^{(p)}) - \nabla^2_E H^{(p)}_{N}(\bm{x};{J'}^{(p)})}_{op}
        \leq
        \frac{1}{\sqrt{N}}
        (p(p-1))
        \left(
        \sum_{i_1,\ldots,i_p=1}^N 
        \left(J_{i_1,\ldots,i_p}^{(p)} - {J'}_{i_1,\ldots,i_{p}}^{(p)}
        \right)^2
        \right)^{1/2}.
        \label{eq:J difference conclusion_p}
    \end{align}
    Combining \eqref{eq:Hoffman--Wielandt} and \eqref{eq:J difference conclusion_p}, we obtain 
    \begin{align}
        &\abs*{F(J^{\leq p_0}) - F({J'}^{\leq p_0})}
        \nonumber \\
        &\leq 
        \frac{1}{\sqrt{N}}
        \sum_{p=2}^{p_0} \gamma_p p(p-1)
        \left(
        \sum_{i_1,\ldots,i_p=1}^N 
        \left(J_{i_1,\ldots,i_p}^{(p)} - {J'}_{i_1,\ldots,i_{p}}^{(p)}
        \right)^{1/2}
        \right)
        \nonumber \\
        &\leq
        \frac{1}{\sqrt{N}}
        \underbrace{
        \left(
        \sum_{p=2}^{\infty} (\gamma_p p(p-1))^2
        \right)^{1/2}
        }_{=C_\gamma}
        \left(
        \sum_{p=2}^{p_0}\sum_{i_1,\ldots,i_p=1}^N 
        \left(
        J_{i_1,\ldots,i_p}^{(p)} - {J'}_{i_1,\ldots,i_{p}}^{(p)}
        \right)^2
        \right)^{1/2} \nonumber \\
        &=
        \frac{1}{\sqrt{N}}
        C_\gamma
        \left(
        \sum_{p=2}^{p_0}\sum_{i_1,\ldots,i_p=1}^N 
        \left(
        J_{i_1,\ldots,i_p}^{(p)} - {J'}_{i_1,\ldots,i_{p}}^{(p)}
        \right)^2
        \right)^{1/2}
        \nonumber \\
        &=
        \frac{1}{\sqrt{N}}
        C_\gamma
        \norm*{J^{\leq p_0} - {J'}^{\leq p_0}}_{\R^{N^2}\times\cdots\times\R^{N^{p_0}}}, \label{eq:F Lip const}
    \end{align}
    where $\norm{\cdot}_{\R^{N^2}\times\cdots\times\R^{N^{p_0}}}$ is the Euclidean norm on $\R^{N^2}\times\cdots\times\R^{N^{p_0}}$.
    Therefore, by Condition~\hyperlink{ass:LS}{(LS)}, \eqref{eq:F Lip const} and Proposition~2.3~in~\cite{ledoux_concentration_1999}, we have
    \begin{align*}
        \PR{
        \abs{
        \int_\R f(y)\dd{\mu_{N,p,\bm{x}}(y)} - 
        \int_\R f(y)\dd{\bar{\mu}_{N,p,\bm{x}}(y)
        }
        }
        \geq \delta
        }
        \leq 2 \exp(-\frac{\delta^2}{2KC_\gamma^2} N),
    \end{align*}
    where $K$ appeared in Condition~\hyperlink{ass:LS}{(LS)} and \eqref{eq:F Lip const}.
\end{proof}

\subsection{Proof of Proposition~\ref{prop:sup edge control}}
\label{sec:proof sup edge control}
    Fix $c>0$, $\varepsilon>0$ and $\bm{x}\in\bm{B}^N$ and $\rho_{\bm{x}}=\norm{\bm{x}}/N$. Define $f:\R\rightarrow\R$ by
    \begin{align*}
        f(y)
        =
        \begin{cases}
            1, & y\in(-\infty,-2\xi''(\rho_{\bm{x}})^{1/2}+\frac{\varepsilon}{2}), \\
            -\frac{2}{\varepsilon}(y+2\xi''(\rho_{\bm{x}})^{1/2}-\frac{\varepsilon}{2})+1, & y\in [-2\xi''(\rho_{\bm{x}})^{1/2}+\frac{\varepsilon}{2},-2\xi''(\rho_{\bm{x}})^{1/2}+\varepsilon], \\
            0, & y\in (-2\xi''(\rho_{\bm{x}})^{1/2}+\varepsilon,\infty),
        \end{cases}
    \end{align*}
    Note that $f$ is bounded Lipschitz with $\abs*{f}_\Lip=2/\varepsilon$ and 
    \begin{align}
        \Ind_{(-\infty,-2\xi''(\rho_{\bm{x}})^{1/2}+\frac{\varepsilon}{2}]}
        \leq f
        \leq 
        \Ind_{(-\infty,-2\xi''(\rho_{\bm{x}})^{1/2}+\varepsilon]}. \label{eq:f sandwich}
    \end{align}
    Thus, \eqref{eq:f sandwich} implies
    \begin{align}
        \mu_{N,\bm{x}}((-\infty,-2\xi''(\rho_{\bm{x}})^{1/2}+\varepsilon])
        \geq
        \int_{\R} f(y)\dd{\mu_{N,\bm{x}}(y)}. \label{eq:1<=f}
    \end{align}
    Also, \eqref{eq:f sandwich} yields
    \begin{align}
        &\int_{\R} f(y)\dd{\bar{\mu}_{N,p,\bm{x}}(y)}
        \nonumber \\
        &=
        \int_{\R} f(y)\dd{\mu_{sc,2\xi''(\rho_{\bm{x}})^{1/2}}(y)} 
        \nonumber \\
        &-
        \int_{\R} f(y)\dd{\mu_{sc,2\xi''(\rho_{\bm{x}})^{1/2}}(y)}
        +
        \int_{\R} f(y)\dd{\mu_{sc,2\xi_p''(\rho_{\bm{x}})^{1/2}}(y)}
        \nonumber \\
        &-
        \int_{\R} f(y)\dd{\mu_{sc,2\xi_p''(\rho_{\bm{x}})^{1/2}}(y)}
        +
        \int_{\R} f(y)\dd{\bar{\mu}_{N,p,\bm{x}}(y)}
        \nonumber \\
        &\geq
        \mu_{sc,2\xi''(\rho_{\bm{x}})^{1/2}}
        \left(
        \left(-\infty,-2\xi''(\rho_{\bm{x}})^{1/2}+\frac{\varepsilon}{2}\right]
        \right) 
        \nonumber \\
        &-
        \frac{2}{\varepsilon}
        \norm*{
        \mu_{sc,2\xi''(\rho_{\bm{x}})^{1/2}}
        -
        \mu_{sc,2\xi_p''(\rho_{\bm{x}})^{1/2}}
        }_{LB}
        -
        \frac{2}{\varepsilon}
        \norm*{
        \mu_{sc,2\xi_p''(\rho_{\bm{x}})^{1/2}}
        -
        \bar{\mu}_{N,p,\bm{x}}
        }_{LB}.
        \label{eq:f d_mu low bnd}
    \end{align}
    As $\xi''(q)\leq \xi''(1)$ for all $q\in [0,1]$, we have
    \begin{align}
        \mu_{sc,2\xi''(\rho_{\bm{x}})^{1/2}}
        \left(
        \left(-\infty,-2\xi''(\rho_{\bm{x}})^{1/2}+\frac{\varepsilon}{2}\right]
        \right)
        &=
        \frac{1}{2\pi\xi''(\rho_{\bm{x}})}
        \int_{-2\xi''(\rho_{\bm{x}})^{1/2}}^{-2\xi''(\rho_{\bm{x}})^{1/2}+\frac{\varepsilon}{2}}
        \sqrt{4\xi''(q)-x^2}\dd{x}
        \nonumber \\
        &=
        \frac{1}{2\pi}
        \int_{-2}^{-2+\varepsilon/(2\xi''(\rho_{\bm{x}})^{1/2})}
        \sqrt{4-x^2}\dd{x}
        \nonumber \\
        &\geq
        \frac{1}{2\pi}
        \int_{-2}^{-2+\varepsilon/(2\xi''(1)^{1/2})}
        \sqrt{4-x^2}\dd{x} 
        \nonumber \\
        &= \delta_{\xi,\varepsilon} > 0. \label{eq:c_xi,eps}
    \end{align}

Next, we state and prove an elementary lemma to compare the semicircular laws with different radii.
\begin{lemma}
\label{lem:semicircular comparison}
For all $0< R_1 \leq R_2$, we have
\begin{align}
\norm*{\mu_{sc,R_1} - \mu_{sc,R_2}}_{LB} < 2\left(1-\frac{R_1^2}{R_2^2}\right).
\end{align}
\end{lemma}
\begin{proof}
Let $g\in\Lip_1(\R)$. By the definition of semicircular laws provided in Section~\ref{sec:notation},
\begin{align}
&\abs{
\int_\R g(x)\dd\mu_{sc,R_1}(x) - \int_\R g(x)\dd \mu_{sc,R_2}(x)
} \nonumber \\
&=
\abs{
\int_\R g(x)\left(\frac{2}{\pi R_1^2}\sqrt{R_1^2 - x^2}\Ind_{[-R_1,R_1]} 
- \frac{2}{\pi R_2^2}\sqrt{R_2^2 - x^2}\Ind_{[-R_2,R_2]}
\right)\dd{x}
} \nonumber \\
&\leq 
\bigg|
\int_{-R_1}^{R_1} 
g(x)
\left(\frac{2}{\pi R_1^2}\sqrt{R_1^2 - x^2}
- \frac{2}{\pi R_2^2}\sqrt{R_2^2 - x^2}
\right) \dd{x}
\bigg|
+
\bigg|
\int_{R_2\geq \abs{x}\geq R_1}
g(x) 
\frac{2}{\pi R_2^2}\sqrt{R_2^2 - x^2}
\dd{x}
\bigg|
\nonumber \\
&\eqqcolon
\text{\RomanNumeralCaps 1} + \text{\RomanNumeralCaps 2}. \label{eq:sc decomp}
\end{align}
By telescoping and the Lipschitz continuity of $g$, we have
\begin{align}
\text{\RomanNumeralCaps 1 in \eqref{eq:sc decomp}}
&\leq 
\abs{
\int_{-R_1}^{R_1} 
g(x)
\left(\frac{2}{\pi R_1^2}\sqrt{R_1^2 - x^2}
- \frac{2}{\pi R_2^2}\sqrt{R_1^2 - x^2}
\right) \dd{x}
} \nonumber \\
&+
\abs{
\int_{-R_1}^{R_1} 
g(x)
\left(\frac{2}{\pi R_2^2}\sqrt{R_1^2 - x^2}
- \frac{2}{\pi R_2^2}\sqrt{R_2^2 - x^2}
\right) \dd{x}
} \nonumber \\
&\leq
\left(\frac{2}{\pi R_1^2} - \frac{2}{\pi R_2^2}\right)
\int_{-R_1}^{R_1} 
\sqrt{R_1^2 - x^2} \dd{x}
+
\frac{2}{\pi R_2^2}
\int_{-R_1}^{R_1} 
\bigg(
\sqrt{R_2^2 - x^2} - \sqrt{R_1^2 - x^2}
\bigg)
\dd{x}.
\label{eq:sc decomp.1}
\end{align}
By the Lischitz continuity of $g$,
\begin{align}
\text{\RomanNumeralCaps 2 in \eqref{eq:sc decomp}}
\leq
\frac{2}{\pi R_2^2}
\int_{R_2\geq \abs{x}\geq R_1}
\sqrt{R_2^2 - x^2}
\dd{x}.
\label{eq:sc decomp.2}
\end{align}
Thus, plugging \eqref{eq:sc decomp.1} and \eqref{eq:sc decomp.2} back to \eqref{eq:sc decomp} yields
\begin{align}
\eqref{eq:sc decomp}
&\leq 
\int_{R_2\geq \abs{x}\geq R_1}
\sqrt{R_2^2 - x^2}
\dd{x}
\nonumber \\
&
\leq 
\left(\frac{2}{\pi R_1^2} - \frac{2}{\pi R_2^2}\right)
\cdot
\int_{R_1}^{R_1} \sqrt{R_1^2-x^2}\dd{x}
+
\frac{2}{\pi R_2^2} \left(
\int_{-R_2}^{R_2} \sqrt{R_2^2 - x^2} \dd{x} - \int_{-R_1}^{R_1} \sqrt{R_1^2 - x^2} \dd{x}
\right) \nonumber \\
&=
2 \left( 1- \frac{R_1^2}{R_2^2}\right), \nonumber
\end{align}
and this completes the proof.
\end{proof}    
    By Lemma~\ref{lem:semicircular comparison}, we there exists $p_1\geq 2$ such that 
    \begin{align}
    \norm*{
        \mu_{sc,2\xi''(\rho_{\bm{x}})^{1/2}}
        -
        \mu_{sc,2\xi_p''(\rho_{\bm{x}})^{1/2}}
        }_{LB}
        < 
        \frac{\varepsilon\delta_{\xi,\varepsilon}}{8}, \label{eq:LB1}
\end{align}     
for $p\geq p_1$. Applying \eqref{eq:uniform mu bar} in Proposition~\ref{prop:asymptotic freeness} yields that there exists $N_0=N_0(p_1,\xi)$ such that for $N\geq N_0$, such that 
\begin{align}
        \norm*{
        \mu_{sc,2\xi_p''(\rho_{\bm{x}})^{1/2}}
        -
        \bar{\mu}_{N,p,\bm{x}}
        }_{LB}
        <
        \frac{\varepsilon\delta_{\xi,\varepsilon}}{8} \label{eq:LB2}
\end{align}
Thus, applying \eqref{eq:LB1} and \eqref{eq:LB2} to \eqref{eq:f d_mu low bnd} yields that for $p\geq p_1$ and $N\geq N_0$,
    \begin{align}
    \int_{\R} f(y)\dd{\bar{\mu}_{N,p_1,\bm{x}}(y)}
    \geq \frac{\delta_{\xi,\varepsilon}}{2}.
    \label{eq:mean upp bnd}
    \end{align}

    Take $c'>c$ such that $2e^{-c'N}\leq e^{-cN}$ for $N$ sufficiently large. By Lemma~\ref{lem:hp approx} and Lemma~\ref{lem:concentration}, there exist $p_2\geq 2$ and $\delta_1>0$ such that for all $p\geq p_2$ and $\delta<\delta_1$,
    \begin{align}
        \sup_{\bm{x}\in\bm{B}^N}
        \PR{
        \abs{
        \int_\R f(y)\dd{\mu_{N,\bm{x}}(y)} -
        \int_\R f(y)\dd{\mu_{N,p,\bm{x}}(y)} 
        }
        \geq \delta
        }
        &\leq e^{-c'N}, 
        \label{eq:P1} \\
        \sup_{\bm{x}\in\bm{B}^N}
        \PR{
        \abs{
        \int_\R f(y)\dd{\mu_{N,p,\bm{x}}(y)} - 
        \int_\R f(y)\dd{\bar{\mu}_{N,p,\bm{x}}(y)
        }
        }
        \geq \delta
        }
        &\leq e^{-c'N}.
        \label{eq:P2}
    \end{align}
    On the other hand, on the event that
    \begin{align}
        \abs{
        \int_\R f(y)\dd{\mu_{N,\bm{x}}(y)} -
        \int_\R f(y)\dd{\mu_{N,p,\bm{x}}(y)} 
        }
        &\leq \delta,
        \nonumber \\
        \abs{
        \int_\R f(y)\dd{\mu_{N,p,\bm{x}}(y)} - 
        \int_\R f(y)\dd{\bar{\mu}_{N,p,\bm{x}}(y)
        }
        }
        &\leq \delta, \nonumber 
    \end{align}
    \eqref{eq:mean upp bnd} yields
    \begin{align}
        \int_{\R} f(y)\dd{\mu_{N,\bm{x}}(y)}
        &\geq 
        \int_{\R} f(y)\dd{\bar{\mu}_{N,p_1,\bm{x}}(y)} 
        \nonumber \\
        &- \abs{
        \int_\R f(y)\dd{\mu_{N,\bm{x}}(y)} -
        \int_\R f(y)\dd{\mu_{N,p,\bm{x}}(y)} 
        }
        \nonumber \\
        &- \abs{
        \int_\R f(y)\dd{\mu_{N,p,\bm{x}}(y)} - 
        \int_\R f(y)\dd{\bar{\mu}_{N,p,\bm{x}}(y)
        }
        }
        \nonumber \\
        &
        \geq \frac{\delta_{\xi,\varepsilon}}{2}-2\delta. \label{eq:f d_mu low bnd.2}
    \end{align}
    Thus, by taking $\delta_0=\min\{\frac{\delta_{\xi,\varepsilon}}{8},\delta_1\}$ and $p_0=\max\{p_1,p_2\}$,
    \eqref{eq:P1}, \eqref{eq:P2} and \eqref{eq:f d_mu low bnd.2} yield that for any $\delta<\delta_0$ and $p\geq p_0$,
    \begin{align*}
        &\sup_{\bm{x}\in\bm{B}^N}
        \PR{\mu_{N,\bm{x}}((-\infty,-2\xi''(\rho_{\bm{x}})^{1/2}+\varepsilon])\leq \delta} \\
        &\leq 
        \sup_{\bm{x}\in\bm{B}^N}
        \PR{
        \int_{\R} f(y)\dd{\mu_{N,\bm{x}}(y)}\leq \delta
        } \\
        &\leq 
        \sup_{\bm{x}\in\bm{B}^N}
        \PR{
        \abs{
        \int_\R f(y)\dd{\mu_{N,\bm{x}}(y)} -
        \int_\R f(y)\dd{\mu_{N,p,\bm{x}}(y)} 
        }
        \geq \delta
        } \\
        &+
        \sup_{\bm{x}\in\bm{B}^N}
        \PR{
        \abs{
        \int_\R f(y)\dd{\mu_{N,p,\bm{x}}(y)} - 
        \int_\R f(y)\dd{\bar{\mu}_{N,p,\bm{x}}(y)
        }
        }
        \geq \delta
        } \\
        &\leq 2e^{-c'N} \\
        &\leq e^{-cN},
    \end{align*}
    for $N$ sufficiently large.

\section{Proof of Theorem~\ref{thm:edge}} \label{sec:main proof} 
By \eqref{eq:Hessian Lip}~of~Corollary~\ref{prop:Lipschitz} and the fact that 
\begin{align*}
    \norm*{\nabla^2 H_N(\bm{x})-\nabla^2 H_N(\bm{y})}_{op}
    \leq \norm*{\nabla^2_E H_N(\bm{x})-\nabla^2_E H_N(\bm{y})}_{op}, \quad \bm{x},\bm{y}\in\bm{B}^N,
\end{align*}
there exist $c_1>0$ and $R>0$ such that
\begin{align}
    &\PR
    {
    \exists \bm{x}\in\bm{B}^N : 
    \mu_{N,\bm{x}}((-\infty,-2\xi''(\rho_{\bm{x}})^{1/2}+\varepsilon])
    \leq \delta
    } \nonumber \\
    &\leq 
    \PP
    \Big(
    \Big\{\exists \bm{x}\in\bm{B}^N : 
    \mu_{N,\bm{x}}((-\infty,-2\xi''(\rho_{\bm{x}})^{1/2}+\varepsilon])
    \leq \delta
    \Big\} \nonumber \\
    &\quad\quad\quad \cap
    \Big\{
    \forall \bm{x},\bm{y}\in \bm{B}^N : 
    \norm*{\nabla^2 H_N(\bm{x})-\nabla^2 H_N(\bm{y})}_{op} 
    < \frac{R}{\sqrt{N}} \norm{\bm{x}-\bm{y}}
    \Big\}
    \Big)
    + e^{-c_1 N}, \label{eq:Lipschitz bnd}
\end{align}
where $\rho_{\bm{x}}=\norm{\bm{x}}/N$ in the first line above.
Thus, it remains to show that the first term of \eqref{eq:Lipschitz bnd} decays exponentially to zero, which is done by the following discretization argument.

Let $\bm{T}_N$ be a $(\varepsilon\sqrt{N}/2R)$-net
of $\bm{B}^N$. On the event 
\begin{align*}
    \Big\{
    \forall \bm{x},\bm{y}\in \bm{B}^N : 
    \norm*{\nabla^2 H_N(\bm{x})-\nabla^2 H_N(\bm{y})}_{op} 
    < \frac{R}{\sqrt{N}} \norm{\bm{x}-\bm{y}}
    \Big\},
\end{align*}
if there exists $\bm{x}\in\bm{B}^N$ such that
\begin{align*}
    \mu_{N,\bm{x}}((-\infty,-2\xi''(\rho_{\bm{x}})^{1/2}+\varepsilon])\leq \delta,
\end{align*}
then there exists $\bm{y}\in \bm{T}_N$ such that 
\begin{align*}
    \mu_{N,\bm{y}}((-\infty,-2\xi''(\rho_{\bm{y}})^{1/2}+2\varepsilon])\leq \delta.
\end{align*}

Therefore, by the discretization argument provided in the previous paragraph,
the first term of \eqref{eq:Lipschitz bnd} is bounded from above by
\begin{align}
    &\PP
    \Big(
    \exists \bm{y}\in \bm{T}_N : 
    \mu_{N,\bm{y}}((-\infty,-2\xi''(\rho_{\bm{y}})^{1/2}+2\varepsilon])\leq \delta
    \Big)
    \nonumber \\
    &\leq 
    \sum_{\bm{y}\in \bm{T}_N} 
    \PP
    \Big(
    \mu_{N,\bm{y}}((-\infty,-2\xi''(\rho_{\bm{y}})^{1/2}+2\varepsilon])\leq \delta
    \Big) 
    \label{eq:uni bound} \\
    &\leq \abs{\bm{T}_N}\times e^{-c_2 N} \label{eq:uni bound.2} \\
    &\leq \max\left
    \{
    \left(\frac{6R}{\varepsilon}\right)^N,1
    \right\} 
    e^{-c_2 N}, \label{eq:uni bound.3}
\end{align}
    where \eqref{eq:uni bound} follows from the uniform bound, \eqref{eq:uni bound.2} follows from Proposition~\ref{prop:sup edge control} 
    and \eqref{eq:uni bound.3} is a standard estimate (cf. Corollary~4.2.13~in~\cite{vershynin_high-dimensional_2018} for a proof).
    Since $\bm{T}_N$ is arbitrary chosen, by taking $c_2 > \log(10R/\varepsilon)$, \eqref{eq:uni bound.3} yields that the first term is bounded above by $0.6^N$, and the proof is completed.

\paragraph{Acknowledgments.}
I gratefully acknowledge support from Eliran Subag’s research grants: the Israel Science Foundation (grant No. 2055/21) and the European Research Council (ERC) under the European Union’s Horizon Europe research and innovation programme (grant agreement No. 101165541). I also thank him for suggesting the project and for numerous insightful discussions. I am also grateful to Mireille Capitaine, Guillaume C\'{e}bron, and Ofer Zeitouni for stimulating discussions. Finally, I appreciate 
Anouar Kouraich for his comments on the first version of this paper on arXiv.

\providecommand{\bysame}{\leavevmode\hbox to3em{\hrulefill}\thinspace}
\providecommand{\MR}{\relax\ifhmode\unskip\space\fi MR }
\providecommand{\MRhref}[2]{%
  \href{http://www.ams.org/mathscinet-getitem?mr=#1}{#2}
}
\providecommand{\href}[2]{#2}

\end{document}